\documentclass[11pt,a4paper]{article}

\usepackage[utf8]{inputenc}
\usepackage{color}
\usepackage{wrapfig}
\usepackage{placeins}
\usepackage{subfigure}
\usepackage{tabu}
\usepackage{fullpage}
\usepackage{graphicx}
\usepackage{tikz}
\usepackage{pgfplots}
\usepackage{epstopdf}
\usepackage{epsfig}
\usepackage[colorlinks=true,bookmarks=false,linkcolor=blue,urlcolor=blue,citecolor=blue,breaklinks=true]{hyperref}
\usepackage{tikz}
\usepackage{natbib}
\usepackage{tikz-qtree}
\usepackage{eurosym}
\usepackage{amsmath}
\numberwithin{equation}{section}
\usepackage{amssymb}
\usepackage{amsthm}
\usepackage{mathrsfs}
\usepackage{blkarray}
\usepackage{dsfont}% use $\mathds{1}$

\newcommand{\sigmamin}{\sigma_{\operatorname{min}}}

\newcommand{\D}{\mathrm{D}}

\newcommand{\N}{\mathbb{N}}

\newcommand{\R}{\mathbb{R}}
\newcommand{\T}{\mathrm{T}}

\newcommand{\inner}[2]{\left\langle{#1},{#2}\right\rangle}

\newcommand{\M}{\mathcal{M}}

\newcommand{\calO}{\mathcal{O}}

\newcommand{\grad}{\mathrm{grad}}
\newcommand{\Hess}{\mathrm{Hess}}
\newcommand{\hess}{\nabla^2}
\newcommand{\dist}{\mathrm{dist}}

\newcommand{\Rn}{{\mathbb{R}^n}}

\newcommand{\Rnn}{{\mathbb{R}^{n \times n}}}

\newcommand{\lambdamin}{\lambda_\mathrm{min}}

\newcommand{\Id}{\operatorname{Id}} % need to be distinguishable from identity matrix

\usepackage{pifont}

\usepackage{relsize}
\usepackage{rotating}

\usepackage{bm}

\DeclareMathOperator*{\argmin}{arg\,min}

\usepackage{algorithm}%
\usepackage{algpseudocode}
\usepackage{todonotes}
\usepackage{makecell}

\newcommand{\norm}[1]{\left\|#1\right\|}

\newtheorem{theorem}			     {Theorem}	[section]
\newtheorem{proposition}[theorem]	 {Proposition}
\newtheorem{corollary}	  [theorem]	 {Corollary}
\newtheorem{lemma}	      [theorem]  {Lemma}
\newtheorem{definition}	        	 {Definition}[section]
\newtheorem{assumption} {A\ignorespaces}
\newtheorem{example}		     {Example}[section]

\newtheorem{remark}			  {Remark}	[section]

\usepackage{mathtools}
\mathtoolsset{showonlyrefs=false}

\newcommand{\aref}[1]{\hyperref[#1]{A\ref{#1}}}

\usepackage{authblk}
\usepackage{tcolorbox}
\usepackage{comment}
\title{Riemannian trust-region methods for strict saddle functions with complexity guarantees}
\author{Florentin Goyens\thanks{\textsc{LAMSADE}, \textsc{CNRS}, Universit\'e Paris Dauphine-PSL, Place du Mar\'echal de Lattre de Tassigny, 75016 Paris, France. Emails: \url{goyensflorentin@gmail.com} (corresponding author), ~\url{clement.royer@lamsade.dauphine.fr}.} ~and Cl\'ement W. Royer$^*$}

\date{\today}

\newcommand{\kplus}{_{k+1}}
\newcommand{\inv}{^{-1}}
\newcommand{\onehalf}{\dfrac{1}{2}}
\newcommand{\calS}{\mathcal{S}}
\newcommand{\varepsilong}{\varepsilon_g}
\newcommand{\varepsilonH}{\varepsilon_H}
\newcommand{\Deltamin}{\Delta_\mathrm{min}}
\newcommand{\Deltabar}{\bar\Delta}
\newcommand{\TxM}{\mathrm{T}_x \mathcal{M}}

\newcommand{\ualpha}{\underline{\alpha}}
\newcommand{\ubeta}{\underline{\beta}}
\newcommand{\ugamma}{\underline{\gamma}}
\newcommand{\udelta}{\underline{\delta}}

\newcommand{\cR}{\kappa_R}
\newcommand{\cns}{\nu_S}
\newcommand{\cdxRs}{\kappa_S}
\newcommand{\cH}{\kappa_H}
\newcommand{\lmax}{\cH}
\newcommand{\cDelta}{c_{\Delta}}
\newcommand{\cquad}{c_Q}

\newcommand{\cDeltain}{\tilde{c}_{\Delta}}
\newcommand{\cquadin}{\tilde{c}_Q}
\newcommand{\Deltaminin}{\widetilde{\Delta}_{\min}}
\newcommand{\Cinex}{\tilde{C}}
\newcommand{\calSin}{\tilde{\calS}}

\newcommand{\smallstep}{\nu_R}

\newcommand{\lHhat}{\hat{L}_H}

\newcommand{\Ralpha}{\mathcal{R}_1}
\newcommand{\Rbeta}{\mathcal{R}_2}
\newcommand{\Rgamma}{\mathcal{R}_3}
\newcommand{\Ialpha}{\mathcal{S}_1}
\newcommand{\Ibeta}{\mathcal{S}_2}
\newcommand{\Igamma}{\mathcal{S}_3}
\newcommand{\kmax}{J_{\mathrm{CG}}}
\newcommand{\Nmeo}{J_{\mathrm{MEO}}}

\newcommand{\pullback}{\hat{f}_k}

\begin{document}
\maketitle

%\tableofcontents

\begin{abstract}
	The difficulty of minimizing a nonconvex function is in part explained by the presence of saddle points. This slows down optimization algorithms and impacts worst-case complexity guarantees. However, many nonconvex problems of 
	interest possess a favorable structure for optimization, in the sense that saddle points can be escaped efficiently by appropriate algorithms. This \emph{strict saddle} property has been extensively 
	used in data science to derive good properties for first-order algorithms, such as convergence to 
	second-order critical points. However, the analysis and the design of second-order algorithms in the strict 
	saddle setting have received significantly less attention.
	
	In this paper, we consider second-order trust-region methods for a class of strict saddle functions 
	defined on Riemannian manifolds. These functions exhibit (geodesic) strong convexity around minimizers 
	and negative curvature at saddle points. We show that the standard trust-region method with exact 
	subproblem minimization finds an approximate local minimizer in a number of iterations that depends 
	logarithmically on the accuracy parameter, which significantly improves known results for general 
	nonconvex optimization. We also propose an inexact variant of the algorithm that explicitly leverages the 
	strict saddle property to compute the most appropriate step at every iteration. Our bounds for the inexact 
	variant also improve over the general nonconvex case, and illustrate the benefit of using strict saddle 
	properties within optimization algorithms.\\
	
	
\end{abstract}

%%%%%%%%%%%%%%%%%%%%%%%%%%%%%%%%%%%%%%%%%%%%%%%%%%%%%%%%%%%%%%%%%%%%%%%%%%%%%%%%%%%%%%%%%%%%%%%%%%%%%%%%%%%%%
\section{Introduction}
\label{sec:intro}
%%%%%%%%%%%%%%%%%%%%%%%%%%%%%%%%%%%%%%%%%%%%%%%%%%%%%%%%%%%%%%%%%%%%%%%%%%%%%%%%%%%%%%%%%%%%%%%%%%%%%%%%%%%%%

% Setup
We consider the optimization problem
\begin{equation}\label{eq:P}\tag{P}
\min_{x\in \M} f(x)
\end{equation}
where $\M$ is an $n$-dimensional Riemannian manifold, and $f\colon \M\to \R$ is twice continuously differentiable and nonconvex. 
A popular way to solve Problem~\eqref{eq:P} is to use Riemannian optimization techniques, that use
differential geometry to generalize unconstrained optimization methods to the Riemannian 
setting~\citep{absil2008,boumal2023}. Theoretical guarantees for such methods have historically focused 
on the behavior close to minimizers (local convergence). These results usually rely on the 
objective function being convex (or strongly convex) around minimizers, thereby enabling the 
derivation of local convergence rates~\citep{absil2008,nocedalNumericalOptimization2006}. 

Meanwhile, the past decade has seen a growing interest in global convergence results for nonconvex 
optimization, where one quantifies the rate of convergence towards a stationary point independently of the 
starting point~\citep{cartis2022evaluation}. These rates can be stated in the form of complexity 
results, which bound the number of iterations necessary to satisfy approximate first- or second-order necessary conditions for optimality. Second-order stationary points for Problem~\eqref{eq:P} have a zero Riemannian gradient and positive semidefinite Riemannian Hessian:
\begin{align}\label{eq:socp}
\grad f(x) &= 0 & \text{and} && \lambda_{\min}\left(\Hess f(x)\right) &\geq 0,
\end{align}
where $\lambdamin(\cdot)$ is the smallest eigenvalue of a symmetric operator.
% and $\grad f$ and $\Hess f$ refer to the Riemannian gradient and Riemannian Hessian of $f$ on $\M$. 
Given positive tolerances $(\varepsilon_g, \varepsilon_H)$, complexity results bound the cost of satisfying an approximate version of~\eqref{eq:socp}, given by
%Complexity guarantees apply to algorithms which target an approximate version of~\eqref{eq:socp},
% bound the number of iterations necessary to satisfy an approximate version of~\eqref{eq:socp},
%  which for positive tolerances $(\varepsilon_g, \varepsilon_H)$, is given by
\begin{align}\label{eq:target}
\norm{\grad f(x)} &\leq \varepsilon_g & \text{and} && \lambda_{\min}\left(\Hess f(x)\right) &\geq - \varepsilon_H.
\end{align}
In the unconstrained or Euclidean setting (i.e., when $\M=\R^n$), it is well established that classical 
second-order trust-region methods~\citep{conn2000trust} reach an iterate satisfying~\eqref{eq:target} in at most 
$\calO(\max(\varepsilon_g^{-2}\varepsilon_H^{-1},\varepsilon_H^{-3}))$ iterations~\citep{cartisSecondOrder2012}. 
Although this complexity can be improved to $\calO(\max(\varepsilon_g^{-2},\varepsilon_H^{-3}))$ without 
changing the essence of the algorithm~\citep{curtisConcise2018,grattonDecoupled2020}, the resulting bound 
remains suboptimal among a large class of second-order methods~\citep{cartisWorstCase2019}. Indeed, 
techniques such as cubic regularization enjoy a 
$\calO(\max(\varepsilon_g^{-3/2},\varepsilon_H^{-3}))$ complexity bound, that strictly 
improves over standard trust-region methods and is optimal among the class of 
second-order algorithms. Similar bounds were obtained for the Riemannian counterparts of trust-region methods~\citep{boumal2019global} and cubic regularization~\citep{boumal2021cubic}. Modifications of the trust-region scheme have been proposed to achieve the optimal complexity of cubic regularization~\citep{curtisTrustRegionAlgorithm2017,
curtis2021trust}. 

These worst-case results are pessimistic in nature and do not reflect the good behaviour of second-order methods on many practical problems. In an effort to reconcile theoretical guarantees with practical performances, it becomes necessary to leverage additional structure from the function $f$. Numerous problems of the form~\eqref{eq:P} have the property that the nonconvexity is \emph{benign}, meaning that second-order critical points---Equation~\eqref{eq:socp}---are global minimizers~\citep{sun2015nonconvex,wright2022high}. Data analysis tasks with this property include Burer-Monteiro factorizations of semidefinite programs~\citep{boumal2020deterministic,luo2022nonconvex}, phase retrieval~\citep{sun2018geometric}, matrix completion and factorization~\citep{ge2016matrix,gongguo2019nonconvex}, dictionary learning~\citep{sun2017complete1,qu2019analysis} and others.

Benign nonconvexity implies that the Hessian possesses a negative eigenvalue at every saddle point. This \emph{strict saddle} property allows first- and second-order methods to provably avoid saddle points and converge towards minimizers. First-order methods escape strict saddle points almost 
surely~\citep{lee2019first}, and complexity bounds can even be derived for randomized first-order 
techniques, in both the Euclidean and Riemannian setting~\citep{criscitiello2019efficiently,sun2019escaping}. 
In addition, second-order methods, that leverage directions of negative curvature of the Hessian, escape 
strict saddle points by design, and are thus particularly suitable 
for strict saddle problems~\citep[Chapter 9]{wright2022high}.

Adaptations of complexity analysis to strict saddle problems have recently begun to appear in the literature. 
%The question of adapting complexity analyses to problems satisfying the strict saddle property has received limited attention in the literature.
On one hand, complexity results were established for specific instances 
satisfying a strict saddle property, such as phase retrieval~\citep{sun2018geometric} or dictionary 
learning~\citep{sun2017completeb}. More recently, \citet{oneill2023linesearch} considered low-rank matrix 
optimization problems under a strict saddle property, and designed a line-search method that made explicit use 
of the strict saddle structure. In these works, the analysis is tailored to specific problems, and its 
generalization to a broader strict saddle setting is not straightforward.

On the other hand, general analyses based on dividing the feasible set into regions of interest yielded 
complexity bounds that improved over the general nonconvex setting, in the 
sense that the dependencies with respect to $\varepsilon_g$ and $\varepsilon_H$ were only logarithmic rather 
than polynomial~\citep{paternain2019newton,curtis2021regional}. \citet{carmon2018accelerated} showed that an 
accelerated gradient technique tailored to nonconvex problems would enjoy improved complexity when applied to 
a function satisfying the strict saddle property. These general results apply to unconstrained strict saddle problems, and do not cover optimization problems on manifolds, a popular source of strict saddle 
problems~\citep{wright2022high}.

\subsection*{Contributions and outline}

In this work, we analyze a trust-region framework for minimizing strict saddle functions over Riemannian 
manifolds. The strict saddle problems we consider are strongly convex near minimizers, which leads to connections with Riemannian optimization of geodesically strongly convex functions. In particular, 
we leverage local convergence results for Newton's method in order to derive complexity results for 
our framework. We show that the standard trust-region method~\citep{absil2007trust} with exact subproblem minimization applied to a strict saddle function benefits from improved complexity guarantees compared to the general nonconvex setting. Indeed, our complexity bound possesses a logarithmic dependency in the optimality tolerances, 
thanks to the local quadratic convergence of the method, which improves over polynomial dependencies 
from the general case. We also derive similar results for an inexact version of our algorithm based on inexact 
solutions of the trust-region subproblem, that makes explicit use of the strict saddle structure. Our analysis builds on recent advances in the complexity of (Euclidean) trust-region methods by relying on iterative linear 
algebra routines. This yields complexity bounds in terms of iterations as well as Hessian-vector products.

 To the best of our knowledge, we provide the first strict saddle analysis of a generic second-order trust-region method, and the first strict saddle analysis that applies to a generic manifold $\M$. All our results apply naturally to the unconstrained case $\M=\Rn$. Overall, our results advocate for further use of the strict saddle structure in the design and analysis of 
nonconvex optimization methods. 

The rest of the paper is organized as follows. In Section~\ref{sec:background}, we describe the class of 
strict saddle functions on Riemannian manifolds that we investigate throughout the paper. This is prefaced by 
background material on Riemannian optimization and geodesic convexity. In Section~\ref{sec:exact_RTR}, we 
analyze the global complexity of the Riemannian trust-region with exact subproblem minimization. This is a 
well-known algorithm for which we show an improved complexity when applied to strict saddle functions. In 
Section~\ref{sec:inexact}, we design a new Riemannian trust-region method with inexact subproblem 
minimization that uses landscape parameters to compute directions which are appropriate for the local 
landscape. The guarantees for the inexact algorithm account for the cost of solving the subproblem. 

%%%%%%%%%%%%%%%%%%%%%%%%%%%%%%%%%%%%%%%%%%%%%%%%%%%%%%%%%%%%%%%%%%%%%%%%%%%%%%%%%%%%%%%%%%%%%%%%%%%%%%%%%%%%%
\section{Strict saddle functions on Riemannian manifolds}
\label{sec:background}
%%%%%%%%%%%%%%%%%%%%%%%%%%%%%%%%%%%%%%%%%%%%%%%%%%%%%%%%%%%%%%%%%%%%%%%%%%%%%%%%%%%%%%%%%%%%%%%%%%%%%%%%%%%%%

In this section, we define a class of strict saddle functions on Riemannian manifolds. We first present background material on Riemannian optimization in Section~\ref{ssec:manifolds},
with a focus on retractions. We then discuss the notion of geodesic strong convexity in
Section~\ref{ssec:geodesiccvx}, which plays a role in our definition of strict saddle functions. This
definition is provided along with several examples in Section~\ref{ssec:defstrict}. 

%%%%%%%%%%%%%%%%%%%%%%%%%%%%%%%%%%%%%%%%%%%%%%%%%%%%%%%%%%%%%%%%%%%%%%%%%%%%%%%%%%%%%%%%%%%%%%%%%%%%%%%%%%%%%
\subsection{Retractions and derivatives on Riemannian manifolds}
\label{ssec:manifolds}

Recall that problem~\eqref{eq:P} considers the minimization of a smooth function $f$ over a Riemannian
manifold $\M$. We cover the basic ideas that allow to build feasible algorithms for~\eqref{eq:P}.

At every $x \in \M$, the linear approximation of the manifold $\M$ is called the \emph{tangent space}, written $\T_x \M$. Each tangent space is equipped with an inner product $\inner{\cdot}{\cdot}_x$, which defines the norm of a tangent vector as $\norm{v}_x := \sqrt{\inner{v}{v}_x}$ for $v\in \TxM$. (We often write $\inner{\cdot}{\cdot}$ and $\norm{\cdot}$ when the reference point is clear from context.)
% Riemannian optimization algorithms generate a search direction in the form of a tangent vector, an element of the \emph{tangent space} $\T_x \M$. For a given $x$, the tangent space is equipped with an inner product $\inner{\cdot}{\cdot}_x$.
For smooth functions, the metric defines a
\emph{Riemannian gradient} and \emph{Riemannian Hessian} of $f$ at $x \in \M$, which we denote
by $\grad f(x) \in \T_x\M$ and $\Hess f(x)\colon \TxM\to \TxM$, respectively. By contrast, we use the symbols $\nabla$ and
$\hess$ for the gradient and Hessian of a function defined over a Euclidean space. 

Riemannian optimization algorithms use tangent vectors to generate search directions. Following a tangent direction in a straight line may lead outside the manifold, which is undesirable. Therefore, we need a tool to travel on the manifold in a direction prescribed by a tangent vector. This can be done by following the geodesic associated with a tangent vector. On manifolds, geodesics are curves with zero acceleration that generalize the notion of straight line in Euclidean spaces. Formally, a geodesic is a smooth curve
$c\colon I\to \M$ defined on an open interval $I \subset \R$ such that $c''(t)=0$ for all $t\in I$, where
$c''(t)$ is the intrinsic acceleration of $c$~\citep[Chapter 5]{boumal2023}. The exponential map travels along the manifold by following geodesics, but optimization algorithms commonly use first-order approximations of the exponential map, called retractions~\citep[\S 4.1]{absil2008}. A \emph{retraction} at $x$ is a map from the tangent space to the manifold, denoted by $R_x \colon \T_x\M \to \M$. For many manifolds of interest, practical and popular retractions are defined globally~\citep[Chapter 4]{absil2008}. However, the retraction at $x \in \M$ may only be defined locally, in a ball of radius $\varrho(x)>0$ centered around $0_{x}$ in $\T_{x}\M$. In that case the size of the step at $x\in \M$ must be limited to $\varrho(x)$. We discuss this further in Section~\ref{ssec:algoexact}.

Given a retraction, one can lift the function $f$ to the tangent space through the following
composition.
\begin{definition}%[Retraction and pullback]
\label{def:pullback}
	For any $x \in \M$, the \emph{pullback} of $f$ to the tangent space $\T_x \M$ is
	the function $\hat{f}_x \colon \T_x\M \to \R$ defined by
	\begin{equation*}
	\label{eq:pullback}
		 \hat{f}_x(s):= f \circ R_x(s) \ \text{for all } s \in \T_x \M.
	\end{equation*}
\end{definition}

In particular,
given $x \in \M$ and $s \in \T_x\M$, we consider the gradient of the pullback function
$\nabla \hat{f}_x(s) \in \T_x\M$ as well as its Hessian $\nabla^2 \hat{f}_x(s): \T_x\M \to \T_x\M$.
Note the distinction between these derivatives and the Riemannian derivatives of $f$
at $R_x(s)$, denoted by $\grad f(R_x(s))$ and $\Hess f(R_x(s))$. The identities
$\hat{f}_x(0)=f(x)$ and $\nabla \hat{f}_x(0)=\grad f(x)$ hold by
definition~\cite[Proposition 3.59]{boumal2023} and additional assumptions on
the retraction allow to relate the second-order derivatives. To benefit fully from second-order methods, we require that the retraction be a second-order approximation of geodesics.
\begin{assumption}
\label{assu:second_order_retraction}
	The retraction mapping is a \emph{second-order retraction}: for any
	$x \in \M$ and $s\in \T_x\M$, the curve $c: t \in [0,1] \to R_x(ts)$ has zero
	acceleration at $t=0$, that is, $c''(0)=0$.
\end{assumption}
If $R_x$ is a second-order retraction, it holds that
\begin{equation}
\label{eq:Hess2ndordR}
	\hess \hat{f}_x(0) = \Hess f(x) \qquad \forall x \in \M,
\end{equation}
i.e., the Hessian of the pullback function is the Riemannian Hessian of
$f$~\citep[Proposition 5.45]{boumal2023}.

\begin{remark}
\label{re:smoothpullback}
	In this paper, we choose to use a general retraction over the more restrictive exponential map. This requires certain
	smoothness assumptions on the pullback function (see~\aref{assu:hessian_lipschitz}), but has the advantage of resembling the Euclidean 
	setting. Using the exponential map typically leads to a different analysis that relies on parallel transport along geodesics, where the curvature of the manifold appears explicitly~\citep[Section 4]{sun2019escaping,criscitiello2019efficiently}.
\end{remark}

%%%%%%%%%%%%%%%%%%%%%%%%%%%%%%%%%%%%%%%%%%%%%%%%%%%%%%%%%%%%%%%%%%%%%%%%%%%%%%%%%%%%%%%%%%%%%%%%%%%%%%%%%%%%%
\subsection{Geodesic convexity}
\label{ssec:geodesiccvx}

We now provide the key definitions behind geodesic convexity, a concept that generalizes
convexity in Euclidean spaces to Riemannian manifolds. Geodesically convex sets
and functions are defined with respect to geodesics of $\M$ as follows.

%\begin{definition}
%On a Riemannian manifold $\M$ equipped with covariant derivative $\dfrac{\D}{\dt}$, a \emph{geodesic} of $\M$ is a smooth curve
%\emph{geodesic} of the manifold $\MM$ is a smooth curve $c\colon I\to \M$ defined on an open
%interval $I \subset \R$ such that $c''(t)=0$ for all $t\in I$, where $c''(t)$ is the intrinsic acceleration
%of $c$ defined by the covariant derivative
%\end{definition}
%A subset $S$ of a manifold $\M$ is geodesically convex if every pair of point in $S$ can be connected by a geodesic contained in $S$.

\begin{definition}
\label{def:geocvxset}
	A subset $S$ of $\M$ is geodesically convex if, for every $x,y\in S$,
	there exists a geodesic segment $c\colon [0,1] \to \M$ such that $c(0) = x,~c(1)=y$ and
	$c(t)$ is in $S$ for all $t\in [0,1]$.
\end{definition}
A function is geodesically convex on $S \subset \M$ if it is convex in the usual sense along all geodesics on $S$.
\begin{definition}
\label{def:geocvxfun}
	Given a subset $S$ of $\M$, the function $f\colon \M\to \R$ is geodesically convex on $S$ (resp. geodesically strongly
	convex) if	$S$ is geodesically convex and for every geodesic $c:[0,1] \to \M$ such that $c(0)\neq c(1)$
	and $c([0,1]) \subset S$, the function $f\circ c\colon [0,1]\to \R$ is convex (resp. strongly convex).
\end{definition}
For smooth functions, geodesic strong convexity is
determined by the eigenvalues of the Riemannian Hessian.
\begin{proposition}[Theorem 11.23 in~\citep{boumal2023}]
\label{prop:geostrcvxeigs}
	A function $f\colon \M\to \R$ is geodesically $\gamma$-strongly convex on the set $S\subset \M$
	if $S$ is a geodesically convex set and $\lambdamin \left(\Hess f(x)\right) \geq \gamma$ for every
	$x\in S$.
\end{proposition}
Since we are interested in nonconvex problems, we consider functions that are geodesically strongly convex over a subset of the manifold (near minimizers). Functions with geodesic convexity over the entire manifold have also
been studied, with most interesting applications arising on Hadamard manifolds~\citep{zhang2016first}.

%%%%%%%%%%%%%%%%%%%%%%%%%%%%%%%%%%%%%%%%%%%%%%%%%%%%%%%%%%%%%%%%%%%%%%%%%%%%%%%%%%%%%%%%%%%%%%%%%%%%%%%%%%%%%
\subsection{Strict saddle property}
\label{ssec:defstrict}

We are now ready to define our problem class of interest, robust \emph{strict saddle} functions on $\M$. The definition is based on~\citep{ge2015escaping,sun2015nonconvex}.
\begin{definition}
\label{def:strict_saddle}
Let $f\colon \M\to \R$ be twice differentiable and let $\alpha,\beta,\gamma,\delta$ be positive constants.
The function $f$ is $(\alpha, \beta, \gamma,\delta)$\emph{-strict saddle} if the manifold
	$\M$ satisfies $\M=\Ralpha \cup \Rbeta \cup \Rgamma$, where
	\begin{align*}
		\Ralpha &= \{x\in \M : \norm{\grad f(x)} \geq \alpha\}\\
		\Rbeta &= \{x\in\M: \lambdamin\left(\Hess f(x)\right) \leq -\beta\}\\
		\Rgamma &= \left\{x\in \M: \text{ there exists } x^*\in \M \text{, a local minimizer of $f$ such that } \dist(x,x^*) \leq \delta \text{ and } \right.\\
		&~~~~~~~ \left. f \text{ is geodesically } \gamma\text{-strongly convex over the set }
		\{y\in \M\colon \mathrm{dist}(x^*,y) < 2\delta\}\right\}.
	\end{align*}
\end{definition}

Definition~\ref{def:strict_saddle} has the following interpretation. If $f$ is a strict saddle function
on $\M$, then, at any $x \in \M$, either the norm of the Riemannian gradient is sufficiently large,
the Riemannian Hessian has a sufficiently negative eigenvalue, or $x$ is close to a local
minimum of $f$ on $\M$ and $f$ is geodesically strongly convex in the neighborhood of this local
minimum. Note that the last two cases are mutually exclusive, but that the first case may occur simultaneously
with one of the other two.

\begin{remark}
	Other definitions of strict saddle functions exist in the literature, and the main differences appear in the definition of the region $\Rgamma$, where strong convexity is not always required~\citep{liu2023newton,oneill2023linesearch}. Our definition excludes non-isolated minimizers, where strong convexity cannot hold. Nevertheless, non-isolated minimizers
	can arise due to rotational symmetries in the problem~\citep{wright2022high}. We note that recent work has focused on reformulating such problems on a quotient set induced by the symmetry, leading to problems where minimizers are isolated~\citep{luo2022nonconvex}.	
%	
%	\orange{Under well-understood circumstances, it is possible to reformulate such problems on a \emph{quotient manifold} induced
%	by the symmetry, and this leads to problems where minimizers are isolated. For example, a quotient approach for Burer-Monteiro factorizations of certain
%	semidefinite programs satisfies Definition~\ref{def:strict_saddle}~\citep{luo2022nonconvex}.}
\end{remark}

We conclude this section with two simple examples of strict saddle functions in the sense of
Definition~\ref{def:strict_saddle}. Our first example is a strongly convex function over $\Rn$.

\begin{example}%[Strongly convex function over $\Rn$]
\label{ex:strcvx}
	Let $f:\R^n \to \R$ be a (geodesically) $\gamma$-strongly convex function
	with global minimizer $x^*$. For any $\alpha>0$, $f$ is
	$(\alpha,1,\gamma,\tfrac{2\alpha}{\gamma})$-strict saddle. The region $\Rbeta$ is empty, and for any $x\notin \Ralpha$ (i.e., $\|\nabla f(x)\| < \alpha$), we show that $x\in\Rgamma$. Strong convexity gives
	\[
		\tfrac{\gamma}{2}\|x-x^*\|^2 \le f(x)-f(x^*)\le -\nabla f(x)^\T (x-x^*) \le \|\nabla f(x)\|\|x-x^*\|
		\le \alpha\|x-x^*\|,
	\]
	hence $\|x-x^*\| \le \tfrac{2\alpha}{\gamma}=:\delta$. Clearly, $f$ is $\gamma$-strongly convex on $\{x : \|x-x^*\| < 2\delta\}\subset \Rn$.
\end{example}

Our second example, previously introduced in~\citep{sun2015nonconvex}, illustrates the interest
of the region $\Rbeta$ in the presence of nonconvexity.
\begin{example}
\label{ex:rayleigh}
	Let $\M$ be the unit sphere in $\R^n$, denoted by $\mathbb{S}^{n-1}$, and let
	$f\colon \mathbb{S}^{n-1} \to \R$ be defined by $f(x)=x^\T A x$, where $A \in \Rnn$ is a symmetric
	matrix with eigenvalues $\lambda_1 > \lambda_2 \ge \dots \ge \lambda_{n-1} > \lambda_n$.
	Then, there exists an absolute constant $c>0$ such that $f$
	is $\left( c(\lambda_{n-1}-\lambda_n)/\lambda_1,c(\lambda_{n-1}-\lambda_n),
	c(\lambda_{n-1}-\lambda_n),2c(\lambda_{n-1}-\lambda_n)/\lambda_1\right)$-strict saddle on $\mathbb{S}^{n-1}$.
\end{example}

To end this section, we state our key assumption about Problem~\eqref{eq:P}.
\begin{assumption}\label{assu:ss_R3}
	There exist positive constants $(\alpha,\beta,\gamma,\delta)$ such that the function $f$
	is $(\alpha,\beta,\gamma,\delta)$-strict saddle on the manifold $\M$ and
	$\Rgamma$ is a compact subset of $\M$.
\end{assumption}
The compactness assumption on $\Rgamma$ merely prevents the function from having infinitely many minimizers on $\M$, and is made to simplify the presentation. It is possible to extend our analysis to an unbounded region $\Rgamma$, but this lengthens the argument considerably. Note that Assumption~\ref{assu:ss_R3} holds for both examples above. The boundedness assumption also allows to control the distance between
iterates of our algorithms. 
\begin{lemma}[Lemma 6.32 in~\citep{boumal2023}]
\label{lemma:dist-retraction}
Under~\aref{assu:ss_R3}, there exists positive constants $\cns, \cdxRs$ such that for all $x\in\Rgamma$ and $s\in \T_x\M$, if
	$\norm{s}_x\leq \cns$, then $\dist(x,R_x(s)) \leq \cdxRs\norm{s}_x$, where $\dist(\cdot,\cdot)$ is
	the Riemannian distance on $\M$.
\end{lemma}

%%%%%%%%%%%%%%%%%%%%%%%%%%%%%%%%%%%%%%%%%%%%%%%%%%%%%%%%%%%%%%%%%%%%%%%%%%%%%%%%%%%%%%%%%%%%%%%%%%%%%%%%%%%%%
\section{Riemannian trust-region method with exact subproblem\\ minimization}
\label{sec:exact_RTR}
%%%%%%%%%%%%%%%%%%%%%%%%%%%%%%%%%%%%%%%%%%%%%%%%%%%%%%%%%%%%%%%%%%%%%%%%%%%%%%%%%%%%%%%%%%%%%%%%%%%%%%%%%%%%%

In this section we analyze the classical Riemannian trust-region algorithm (RTR) with exact subproblem minimization. Our goal is to leverage the strict saddle property to obtain better complexity bounds than those
existing for general nonconvex
functions~\citep{boumal2019global}. The algorithm is perfectly standard,
yet the analysis borrows from recent results on Newton-type methods for the Euclidean
setting~\citep{curtis2021trust}. In particular, our improved complexity bounds rely on a good understanding of the local convergence of the algorithm. Note that using the exponential map would simplify the analysis, as noticed for Riemannian cubic regularization~\citep{boumal2021cubic}.

Section~\ref{ssec:algoexact} describes the exact trust-region algorithm, along with key assumptions. Standard
decrease lemmas are provided in Section~\ref{ssec:lemmasexact}. A local convergence analysis of the algorithm
in the region of geodesic strong convexity is provided in Section~\ref{ssec:localcvexact}. This analysis allows to derive our global convergence result, that is established in
Section~\ref{ssec:wccexact}.

%%%%%%%%%%%%%%%%%%%%%%%%%%%%%%%%%%%%%%%%%%%%%%%%%%%%%%%%%%%%%%%%%%%%%%%%%%%%%%%%%%%%%%%%%%%%%%%%%%%%%%%%%%%%%
\subsection{Algorithm and assumptions}
\label{ssec:algoexact}

The Riemannian trust-region method with exact subproblem minimization is described in Algorithm~\ref{algo:exact_RTR}. At every iteration,
a step $s_k$ is computed by minimizing a quadratic model of (the pullback of) the function over the tangent space corresponding
to the current iterate $x_k$. In this section, we assume that the subproblem~\eqref{eq:rtr_subproblem_exact} is solved exactly using
standard approaches~\citep{more1983computing,absil2007trust} (the inexact case is addressed in
Section~\ref{sec:inexact}). The algorithm computes the step $s_k$, then evaluates $f$ at the point $R_{x_k}(s_k) \in \M$ to measure the change in function value. If the function decrease is at least a fraction of the model decrease, the iteration is successful, and the candidate point
$R_{x_k}(s_k)$ becomes the new iterate. The trust-region radius is either unchanged (successful
iteration) or can be increased (very successful iteration). If the iteration is unsuccessful, the algorithm remains at the current iterate and the trust-region radius is decreased.

\begin{algorithm}
\caption{RTR with exact subproblem minimization}\label{algo:exact_RTR}
\begin{algorithmic}[1]
	\State \textbf{Inputs:} Initial point $x_0\in \M$,
	initial and maximal trust-region radii $0<\Delta_0<\bar\Delta$,
	constants $0<\eta_1<\eta_2<1$ and $0<\tau_1<1<\tau_2$.
	\For{$ k = 1, 2, \dots $}
		\State Compute $s_k$ as a solution to the trust-region subproblem
		\begin{equation}\label{eq:rtr_subproblem_exact}
			s_k \in \argmin_{s\in \T_{x_k}\M} m_k(s) \text{ subject to } \norm{s}\leq \Delta_k,
		\end{equation}
		where $m_k$ is the model defined by~\eqref{eq:tr_model}.\\

		\State Compute $\rho_k = \dfrac{f(x_k) - f\left(R_{x_k}(s_k)\right)}{m_k(0) - m_k(s_k)}$ and
		set $x_{k+1} =
				\left\lbrace
					\begin{aligned}
					& R_k(x_k) & \text{ if } \rho_k \ge \eta_1\\
					& x_k & \text{ otherwise.}
					\end{aligned}
				\right.$\\
		\State Set $\Delta_{k+1} =
			\left\lbrace
				\begin{aligned}
					&\min\left(\tau_2\Delta_k,\bar \Delta\right) & \text{ if }   \rho_k >  \eta_2
					&&\text{ [very successful] }\\
					& \Delta_k & \text{ if }   \eta_2 \ge \rho_k \ge  \eta_1
					&&\text{ [successful] }\\
					& \tau_1\Delta_k & \text{ otherwise.}
					&&\text{ [unsuccessful] }
				\end{aligned}
			\right.$
%		\If{$\norm{\grad f(x_k)} \leq \varepsilon_g$
%		and $\lambdamin\left(\Hess f(x_k)\right) \geq - \varepsilon_H$}
%			\State \Return $x_k$.
%		\EndIf
	\EndFor
\end{algorithmic}
\end{algorithm}

The model is a second-order Taylor expansion of the
pullback function $\hat{f}_{x_k}$, namely
\begin{equation}\label{eq:tr_model}
 m_k \colon \T_x\M \to \R \colon s\mapsto  m_k(s) = f(x_k) + \inner{s}{g_k} + \dfrac{1}{2}\inner{s}{H_k s},
\end{equation}
where $g_k = \nabla \hat{f}_{x_k}(0)=\grad f(x_k)$ and $H_k = \hess \hat{f}_{x_k}(0)$, which gives a second-order accurate model. Although we do not consider it here, we belive that our analysis can be extended to approximate second-order accurate models, under the condition that $H_k$ is a suitable approximation of $\hess \pullback(0)$~\citep[Eq. (7.36)]{absil2008}.
When the retraction is second-order (\aref{assu:second_order_retraction}), we have $H_k = \Hess f(x_k)$ by~\eqref{eq:Hess2ndordR}.

In order to derive complexity results for Algorithm~\ref{algo:exact_RTR}, we make a standard Lipschitz-type
assumption on the Hessian of the pullback~\citep{boumal2019global}. Recall that if the retraction at $x_k$ is defined locally, we write $\varrho(x_k)>0$ for the radius of the ball centered around $0_{x_k}$ in $\T_{x_k}\M$ in which it is defined.

\begin{assumption}%[Restricted Lipschitz-type Hessian for pullbacks~\citep{boumal2019global}]
\label{assu:hessian_lipschitz}
	There exists $L_H>0$ such that for all iterates $x_k$ generated by Algorithm~\ref{algo:exact_RTR},
	the pullback $\hat f_k = f\circ R_{x_k}$ satisfies
	\begin{equation}
	\label{eq:hessLip}
		f(R_{x_k}(s))
		\leq f(x_k) + \inner{s}{ \grad f(x_k)} + \dfrac{1}{2} \inner{s}{\hess\pullback(0)[s]}
		+ \dfrac{L_H}{6}\norm{s}^3.
	\end{equation}
	for all $s \in \T_{x_k}\M$ such that $\norm{s} \leq  \varrho(x_k)$. We further assume $\Delta_k\leq \varrho(x_k)$, so that the property~\eqref{eq:hessLip} holds in the entire trust region produced by Algorithm~\ref{algo:exact_RTR}. 
\end{assumption}

A simple strategy to ensure $\varrho(x_k) \geq \Delta_k$ in the assumption above is to set $\bar \Delta $ below $\inf_{x\in \M\colon f(x) \leq f(x_0)} \varrho(x)$, which is positive if the injectivity radius of the manifold is positive~\citep[Remark 2.2]{boumal2019global}. Throughout we work implicitly under the assumption that this is satisfied. We additionally make the following assumption on the Hessian operators considered throughout the algorithm.
\begin{assumption}
\label{assu:bounded_hessian}
	There exists $\cH>0$ such that for all iterates $x_k$ generated by Algorithm~\ref{algo:exact_RTR},
	we have
	\begin{equation}
	\label{eq:boundHk}
 \norm{H_k}:= \sup_{\substack{s\in \T_{x_k}\M \\
 \norm{s}\leq 1}}	|\inner{s}{H_k(s)}| \leq \cH .
	\end{equation}
\end{assumption}

%%%%%%%%%%%%%%%%%%%%%%%%%%%%%%%%%%%%%%%%%%%%%%%%%%%%%%%%%%%%%%%%%%%%%%%%%%%%%%%%%%%%%%%%%%%%%%%%%%%%%%%%%%%%%
\subsection{Preliminary lemmas}
\label{ssec:lemmasexact}

In this section, using standard arguments from the theory of trust-region methods, we bound the model decrease in the regions $\Ralpha,\Rbeta,\Rgamma$ defined by the strict saddle property . We also provide a lower bound on the trust-region radius. 

Our first result handles the case of an iterate with large gradient norm (i.e., in $\Ralpha$).
\begin{lemma}%[Cauchy step decrease]
\label{lemma:cauchy_decrease}
	Under~\aref{assu:ss_R3} and~\aref{assu:bounded_hessian}, consider the $k$th iterate of
	Algorithm~\ref{algo:exact_RTR} and suppose that $x_k \in \Ralpha$.
	Then,
	\begin{equation}\label{eq:cauchy_decrease}
 		m_k(0) -  m_k(s_k) \geq \onehalf \min \left(\Delta_k, \dfrac{\alpha}{\cH}\right)\alpha.
	\end{equation}
\end{lemma}
\begin{proof}
	Define $s_k^C$ as the Cauchy point associated with the trust-region
	subproblem~\eqref{eq:rtr_subproblem_exact}, i.e.
	$s_k^C = -t g_k$ with $t = \argmin_{t \ge 0, \|t\,g_k\| \le \Delta_k} m_k(-t\,g_k)$.
	A straightforward application of~\citet[Lemma 6.15]{boumal2023} gives
	\begin{align*}
		m_k(0) -  m_k(s_k^C)
		&\geq \dfrac{1}{2}\min\left(\Delta_k, \norm{g_k}/\cH\right)\norm{g_k}\\
		&\geq \dfrac{1}{2} \min \left(\Delta_k, \alpha/\cH\right)\alpha.
	\end{align*}
	The desired result follows from the optimality of $s_k$, since
	$m_k(s_k) \le m_k(s_k^C)$.
\end{proof}

Our second result considers an iterate at which the Hessian possesses significant negative curvature
(i.e., in $\Rbeta$).

\begin{lemma}%[Eigenstep decrease]
\label{lemma:eigenstep_decrease}
	Under~\aref{assu:second_order_retraction} and~\aref{assu:ss_R3}, consider the $k$th iterate of
	Algorithm~\ref{algo:exact_RTR} and suppose that $x_k \in \Rbeta$. Then,
	\begin{equation}\label{eq:eigdecrease}
		m_k(0) -m_k(s_k) \geq \dfrac{1}{2} \beta \Delta_k^2 .
	\end{equation}
\end{lemma}
\begin{proof}
	Define $s_k^E = \Delta_k u_k$, where $u_k \in \T_{x_k}\M$ satisfies
	\begin{align*}
		\norm{u_k}_{x_k}&=1, & \inner{g_k}{u_k}_{x_k} &\leq 0 &&\text{ and } &\inner{u_k}{H_k u_k}_{x_k} &\le - \beta.
	\end{align*}
	The vector $s_k^E$---called an eigenstep---exists
	because $x_k \in \Rbeta$. By~\citet[Lemma 6.16]{boumal2023}, it satisfies
	\[
		m_k(0) -m_k(s_k^E) \geq \dfrac{1}{2} \beta \Delta_k^2.
	\]
	The desired result follows from the optimality of $s_k$, as $m_k(s_k) \le m_k(s_k^E)$.
\end{proof}

Our last decrease lemma is based on the strong convexity constant, and proceeds similarly to the previous two results.

\begin{lemma}\label{lemma:convex_decrease}
	Under~\aref{assu:second_order_retraction} and~\aref{assu:ss_R3}, consider the $k$th iterate of
	Algorithm~\ref{algo:exact_RTR} and suppose that $x_k \in \Rgamma$. Then, the step $s_k$ is
	uniquely defined, and satisfies
	\begin{equation}\label{eq:convex_decrease}
		m_k(0) - m_k(s_k) \geq \dfrac{1}{2}\gamma\,\norm{s_k}^2 .
	\end{equation}
\end{lemma}
\begin{proof}
	Since $s_k$ is a solution of the trust-region subproblem~\eqref{eq:rtr_subproblem_exact}, there
	exists $\lambda_k \ge 0$ such that the following optimality conditions hold~\citep[Proposition 7.3.1]{absil2008}:
	\begin{align}
		\left(H_k + \lambda_k\Id\right)s_k &= - g_k\label{eq:opti_1}\\
		\inner{s}{\left(H_k + \lambda_k\Id\right)[s]} &\ge 0 \quad \forall s \in \T_{x_k}\M \label{eq:opti_2}\\
		\norm{s_k} &\leq \Delta_k\\
		\lambda_k(\Delta_k-\norm{s_k}) &=0\label{eq:opti_4}.
	\end{align}
	Moreover, if the inequality in~\eqref{eq:opti_2} is strict for nonzero $s$, then the solution is
	unique. To establish a decrease guarantee for $s_k$, we combine~\eqref{eq:opti_1} and $\lambda_k\geq 0$ to obtain
	\begin{align*}
		m_k(0) - m_k(s_k)
		&= -\inner{s_k}{ g_k} - \dfrac{1}{2}\inner{s_k}{H_k s_k} \\
		&= \inner{s_k}{(H_k + \lambda_k\Id) s_k} - \dfrac{1}{2} \inner{s_k}{ H_k s_k }\\
		&= \dfrac{1}{2}\inner{s_k}{ H_k s_k} + \lambda_k \norm{s_k}^2\\
		&\geq \dfrac{1}{2}\inner{s_k}{ H_k s_k}\\
		&\geq \onehalf \gamma \norm{s_k}^2.
	\end{align*}
	The last line follows from $H_k = \Hess f(x_k)$ and $x_k \in \Rgamma$, i.e., $\inner{s}{H_k s} \ge \gamma \norm{s}^2$ for any $s \in \T_{x_k} \M$. This also implies that
	$H_k+\lambda_k\Id$ is positive definite, hence $s_k$ is uniquely
	defined.
%	 To conclude, we have
%	\[
%		m_k(0)-m_k(s_k) \ge  \dfrac{1}{2}\inner{s_k}{ H_k s_k} \ge \dfrac{\gamma}{2}\norm{s_k}^2.
%	\]
\end{proof}

The three lemmas above, together with the Lipschitz-type assumptions on the pullback,
yield a lower bound on the trust-region radius.
\begin{lemma}
\label{lemma:lower_bound_radius}
Let~\aref{assu:second_order_retraction},~\aref{assu:ss_R3},~\aref{assu:hessian_lipschitz} and \aref{assu:bounded_hessian} hold. Then, for any iteration
	of index $k$, the trust-region radius $\Delta_k$ in Algorithm~\ref{algo:exact_RTR} satisfies
	\begin{equation}\label{eq:lower_bound_radius}
		\Delta_k \geq \Deltamin := \cDelta \min\left(\Delta_0, \alpha^{1/2}, \alpha^{2/3}, \beta, \gamma\right),
	\end{equation}
	where $\cDelta = \min\left(1, \tau_1\sqrt{\dfrac{3(1-\eta_1)}{L_H}},\tau_1 \sqrt[3]{\dfrac{3(1-\eta_1)}{\cH L_H}},
	3\tau_1 \dfrac{(1-\eta_1)}{L_H }\right)$.
%where $\calpha = \tau_1\min \!\left(\dfrac{1}{\cH}, \dfrac{(1-\eta_1)}{L_g + \cH}\right)$ and $ \cbeta = \cgamma = 3\tau_1 \dfrac{(1-\eta_1)}{L_H } $.
\end{lemma}
\begin{proof}
	We begin by showing that if the trust-region radius drops below a certain threshold,
	then the iteration must be successful. We consider the quantity
	\begin{equation}
	\label{eq:1-rho}
		1-\rho_k = 1 - \dfrac{f(x_k) - f(R_{x_k}(s_k))}{m_k(0) - m_k(s_k)} = \dfrac{f(R_{x_k}(s_k))- m_k(s_k)}{m_k(0) - m_k(s_k)}
	\end{equation}
	for the three regions defined by the strict saddle property. First note that for any $x_k\in \M$, a second-order retraction~(\aref{assu:second_order_retraction})
	and Lipschitz continuity of the Hessian~(\aref{assu:hessian_lipschitz}) give the following bound on the numerator of~\eqref{eq:1-rho}:
	\begin{align}
		f(R_{x_k}(s_k))- m_k(s_k)
		&= f(R_{x_k}(s_k)) - f(x_k) - \inner{g_k}{s_k} - \dfrac{1}{2}\inner{s_k}{H_k s_k}\nonumber \\  \label{eq:fkplus1-model}
		&\le \dfrac{L_H}{6}\norm{s_k}^3 \leq \dfrac{L_H}{6}\Delta_k^3 .
	\end{align}
	For $x_k \in \Ralpha$, the denominator of~\eqref{eq:1-rho} satisfies~\eqref{eq:cauchy_decrease}.
	It follows that
	\begin{align*}
		1-\rho_k &\leq \dfrac{L_H \Delta_k^3}{3\min\left(\Delta_k\alpha,\alpha^2/\cH\right)} \leq \dfrac{L_H}{3}\max\left(\dfrac{\Delta_k^2}{\alpha}, \dfrac{\cH \Delta_k^3}{\alpha^2}\right).
	\end{align*}
	As a result, if $x_k \in \Ralpha$ and
	\begin{equation*}
		\Delta_k \leq \min \!\left(\sqrt{\dfrac{3(1-\eta_1)}{L_H}}\alpha^{1/2}, \sqrt[3]{\dfrac{3(1-\eta_1)}{\cH L_H}}\alpha^{2/3}\right),
	\end{equation*}
	then $1-\rho_k \le 1-\eta_1$ and iteration $k$ is successful.
	
	If $x_k\in \Rbeta$, the denominator of~\eqref{eq:1-rho} satisfies~\eqref{eq:eigdecrease}, which we combine with~\eqref{eq:fkplus1-model} to give
	\begin{align*}
		1-\rho_k
		&\leq 	\dfrac{L_H \Delta_k^3}{3\beta\Delta_k^2}
		= \dfrac{L_H }{3\beta}\Delta_k.
	\end{align*}
	Thus, if $x_k\in \Rbeta$ and $\Delta_k \leq 3(1-\eta_1)\beta/L_H$, we have $\rho_k \geq \eta_1$ and iteration $k$ is successful.

	Finally, for $x_k \in \Rgamma$, the upper bound~\eqref{eq:fkplus1-model} together with the model decrease~\eqref{eq:convex_decrease} gives
	\[
		1-\rho_k
		\le \dfrac{L_H \norm{s_k}^3}{3\gamma\norm{s_k}^2}
		\le \dfrac{L_H}{3\gamma}\Delta_k.
	\]
	As a result, if $x_k\in \Rgamma$ and $\Delta_k \leq 3(1-\eta_1)\gamma/L_H$, then $\rho_k \ge \eta_1$ and iteration $k$ is successful.

	Overall, we have shown that the iteration $k$ is successful as long as
	\[
		\Delta_k
		\le
		\min\!\left(\sqrt{\dfrac{3(1-\eta_1)}{L_H}}\alpha^{1/2}, \sqrt[3]{\dfrac{3(1-\eta_1)}{\cH L_H}}\alpha^{2/3},\dfrac{3(1-\eta_1)}{L_H}\beta,\dfrac{3(1-\eta_1)}{L_H}\gamma\right),
	\]
	in which case $\Delta_{k+1} \ge \Delta_k$. It follows from the updating rule on $\Delta_k$ that the trust-region
	radius is lower bounded for any $k \geq 0$:
		\begin{align*}
		\Delta_k &\geq 
		\min\!\left( \Delta_0, \tau_1\sqrt{\dfrac{3(1-\eta_1)}{L_H}}\alpha^{1/2}, \tau_1\sqrt[3]{\dfrac{3(1-\eta_1)}{\cH L_H}}\alpha^{2/3},\tau_1\dfrac{3(1-\eta_1)}{L_H}\beta,\tau_1\dfrac{3(1-\eta_1)}{L_H}\gamma\right)\\
		&\geq \cDelta\min\left(\Delta_0,\alpha^{1/2}, \alpha^{2/3},\beta,\gamma\right). \qedhere
		\end{align*}
\end{proof}

For $x_k\in \Ralpha \cup \Rbeta$, it is straightforward to combine Lemma~\ref{lemma:lower_bound_radius} with the results of Lemmas~\ref{lemma:cauchy_decrease}
and~\ref{lemma:eigenstep_decrease} to guarantee a model decrease that is independent of $k$. For $x_k \in \Rgamma$,
deriving a uniform lower bound on the decrease based on Lemma~\ref{lemma:convex_decrease} is more involved, and is the topic of the next section.

%%%%%%%%%%%%%%%%%%%%%%%%%%%%%%%%%%%%%%%%%%%%%%%%%%%%%%%%%%%%%%%%%%%%%%%%%%%%%%%%%%%%%%%%%%%%%%%%%%%%%%%%%%%%%
\subsection{Region of geodesic strong convexity and local convergence}
\label{ssec:localcvexact}

In this section, we analyze the behavior of Algorithm~\ref{algo:exact_RTR} in the region of strong convexity $\Rgamma$. The global subproblem minimizer is a regularized Newton step---Equation~\eqref{eq:opti_1}---and thus our approach mimics the study of Newton's method applied to
strongly convex functions~\citep{boyd2004}. For sufficiently large gradients, we provide a lower bound on the decrease
achieved by the step; and for small gradients, we show that a local convergence phase begins, during which the iterates converge
quadratically towards a local minimizer of~\eqref{eq:P}. To establish local convergence, we quantify how small the gradient norm needs to be so that the following occurs: the full Newton step is the solution of the subproblem, it is successful, it produces a new iterate that is also in $\Rgamma$, and the sequence of gradient norms converges quadratically towards zero.

Before stating those results, we establish several consequences of the boundedness assumption on $\Rgamma$~(\aref{assu:ss_R3}) that are
helpful in analyzing (regularized) Newton steps, starting with a Lipschitz-type inequality on the gradient of the
pullback.
\begin{lemma}\label{lemma:hessian_lipschtiz_grad_ineq}
	Under~\aref{assu:ss_R3}, there exists $\hat{L}_H>0$ such that for all iterates $x_k\in \Rgamma$ produced
	by Algorithm~\ref{algo:exact_RTR}, we have
	\begin{equation}
	\label{eq:hessgradlip}
		\norm{\nabla \hat f_{k}(s_k) - \nabla \hat f_{k}(0) - \hess \hat f_{k}(0)[s_k]} \leq \dfrac{\hat{L}_H}{2}\norm{s_k}^2.
	\end{equation}
\end{lemma}

\begin{proof}
	The constant $\hat{L}_H$ exists by continuity of the derivatives over $\Rgamma$ (a compact set) and boundedness of the	steps~\citep[Lemma 10.57]{boumal2023}. The compactness of $\Rgamma$ comes from~\aref{assu:ss_R3}, while
	the boundedness of the steps $s_k$ follows from $\norm{s_k} \le \Delta_k \le \Deltabar$.
\end{proof}
%Note that Lemma~\ref{lemma:hessian_lipschtiz_grad_ineq} is stated on the points of interest produced by
%the algorithm, similarly to the Lipschitz-type assumptions~\aref{assu:hessian_lipschitz}.

The next result describes a non-singularity condition for the differential of the retraction on small steps. This allows to relate the gradient of the pullback and the gradient at the next iterate. This property was introduced in~\citep{absil2007trust}, and further analyzed in the context of cubic regularization of Newton's method~\citep{boumal2021cubic}.

\begin{lemma}
\label{lemma:bound_grad_pullback}
	Let $\cR>1$, under~\aref{assu:ss_R3} there exists $\smallstep>0$ such that for any
	$x_k \in \Rgamma$ and $s\in \T_{x_k}\M$ with $\norm{s} \le \smallstep$,
	we have
	\begin{equation}\label{eq:cR}
		\norm{\grad f(R_x(s))}_{x_{k+1}} \leq \cR \norm{\nabla \hat{f}(s)}_{x_{k}}.
	\end{equation}
\end{lemma}
\begin{proof}
	We apply~\citep[Theorem 7]{boumal2021cubic} to $\Rgamma$ as a non-empty compact subset of $\M$. This
	ensures that for any $\cR>1$, there exists a constant $\smallstep>0$ such that, at each $x_k\in \Rgamma$,
	\begin{equation*}
		\norm{s} \le \smallstep \quad \Leftrightarrow \quad \sigmamin(\D R_{x_k}(s)) \geq \frac{1}{\cR},
	\end{equation*}
	where $\D R_{x_k}$ denotes the differential of $R_{x_k}$. The desired conclusion follows by combining
	this result with~\citep[Equation 22]{boumal2021cubic}, which states that
	\begin{equation*}
		\norm{\nabla \hat{f}_{x_k}(s)}
		\geq \sigmamin\left(\D R_{x_k}(s)\right) \norm{\grad f(R_{x_k}(s))}.
		\qedhere
	\end{equation*}
\end{proof}

Combining the previous two lemmas, we bound the change in gradient norm for Newton steps.
\begin{lemma}
\label{lemma:grad_next_step}
	Under~\aref{assu:second_order_retraction} and~\aref{assu:ss_R3}, let $x_k\in \Rgamma$ be
	an iterate produced by Algorithm~\ref{algo:exact_RTR} such that $s_k$ is the
	Newton step, i.e., $H_k s_k=-g_k$, and $\norm{s_k}\leq \smallstep$ where $\smallstep$ is the constant
	from Lemma~\ref{lemma:bound_grad_pullback}. Then, if the iteration is successful, we have
	\begin{equation}
	\label{eq:grad_next_step}
		\norm{\grad f(x_{k+1})}_{x_{k+1}} \leq \cR \dfrac{\hat L_H}{2}\norm{s_k}_{x_k}^2
	\end{equation}
	where $\hat{L}_H$ comes from Lemma~\ref{lemma:hessian_lipschtiz_grad_ineq}.
\end{lemma}
\begin{proof}
	By assumption, $x_{k+1} = R_{x_k}(s_k)$. Using successively Lemma~\ref{lemma:bound_grad_pullback},~\aref{assu:second_order_retraction} and Lemma~\ref{lemma:hessian_lipschtiz_grad_ineq}, we obtain
	\begin{align*}
		\norm{\grad f(x_{k+1})}_{x_{k+1}} &\leq \cR \norm{\nabla \hat f_{x_k}(s_k)}_{x_k} \\
		&= \cR \norm{\nabla \hat f_{x_k}(s_k) - \grad f(x_k) + \grad f(x_k)}_{x_k}\\
		&= \cR \norm{\nabla \hat f_{x_k}(s_k) - \grad f(x_k) - H_k s_k}_{x_k}\\
		&= \cR \norm{\nabla \hat f_{x_k}(s_k) - \nabla \hat f_{x_k}(0)  - \hess \pullback(0) [s_k]  }_{x_k}  \\
		&\leq \cR \dfrac{\hat{L}_H}{2}\norm{s_k}_{x_k}^2. \qedhere
	\end{align*}
\end{proof}
%Note that if the retraction is the exponential map, one shows $\norm{\grad f(x_{k+1})}_{x_{k+1}} \leq \dfrac{L_H}{2}\norm{s_k}^2$ in a similar fashion, where $L_H$ is the Lipschitz constant from~\aref{assu:hessian_lipschitz}.

\begin{remark}
\label{rem:expmapconvexdec}
	Using the exponential map rather than a general retraction significantly simplifies the analysis 
	above. Indeed, with the exponential map, Lemma~\ref{lemma:hessian_lipschtiz_grad_ineq} can be replaced by
	\begin{equation}
	\label{eq:Exp_Lip_bound}
		\norm{P^{-1}_s \grad f(\mathrm{Exp}_x(s)) - \grad f(x) - \Hess f(x)[s]} \leq \dfrac{L_H}{2}\norm{s}^2
	\end{equation}
	where $P^{-1}_s$ is the parallel transport from $\T_{\mathrm{Exp}_x(s)}\M$
	to $\T_x\M$, and $L_H$ is the Lipschitz constant of $\Hess f$ from~\eqref{eq:hessLip}, see~\citep[Corollary 10.56]{boumal2023}.
	As a result, the proof of Lemma~\ref{lemma:grad_next_step} is also simplified, and no
	longer requires the decomposition in short and long steps induced by Lemma~\ref{lemma:bound_grad_pullback}.
\end{remark}

We are equipped to prove a decrease guarantee for Newton steps.

\begin{lemma}%[Newton step decrease]
\label{lemma:decrease_Newton}
	Under the assumptions of Lemma~\ref{lemma:grad_next_step}, we have
	\begin{equation*}
	\label{eq:decrease_Newton}
	m_k(0) - m_k(s_k) \geq \dfrac{\gamma}{\hat{L}_H\,\cR}  \norm{\grad f(x_{k+1})}_{x_{k+1}}.
	\end{equation*}
\end{lemma}
\begin{proof}
	By combining~\eqref{eq:grad_next_step} with~\eqref{eq:convex_decrease}, we obtain
	\begin{align*}
	m_k(0) - m_k(s_k) &\geq \dfrac{\gamma}{2}\norm{s_k}_{x_k}^2\\
	&\geq  \dfrac{\gamma}{\hat{L}_H \cR}\norm{\grad f(x_{k+1})}_{x_{k+1}}.\qedhere
	\end{align*}
\end{proof}

We now turn to local convergence results. Our goal is to show that Newton steps are eventually
accepted by the algorithm, and that they produce iterates with decreasing gradient norm. We begin with a bound on the norm of the subproblem minimizer in $\Rgamma$.
\begin{lemma}
\label{lemma:normskR3}
	Suppose that Algorithm~\ref{algo:exact_RTR} produces an iterate $x_k \in \Rgamma$. Then,
	\begin{align}
	\label{eq:normskR3}
	\norm{s_k} \leq \dfrac{\norm{g_k}}{\gamma}.
	\end{align}
\end{lemma}
\begin{proof}
	The result holds trivially if $\norm{s_k}=0$. Otherwise, using the definition of $\Rgamma$
	together with $\lambda_k \ge 0$ and~\eqref{eq:opti_1}, we get
	\begin{align*}
		\gamma \norm{s_k}^2 &\leq \inner{s_k}{ H_k s_k} \leq \inner{s_k}{ (H_k + \lambda \Id) s_k}
		= - \inner{s_k}{g_k} \leq \norm{s_k}\norm{g_k},
	\end{align*}
	and division by $\norm{s_k}$ gives~\eqref{eq:normskR3}.
\end{proof}
Lemma~\ref{lemma:normskR3} is an elementary identity that we use throughout. We use it in the next proposition to show that iterations in $\Rgamma$ with a small enough gradient are successful. 

\begin{proposition}%[Successful steps]
\label{prop:very_successful_steps}
	Under~\aref{assu:second_order_retraction} and~\aref{assu:hessian_lipschitz}, suppose that
	Algorithm~\ref{algo:exact_RTR} generates $x_k\in \Rgamma$ such that
	\begin{equation}\label{eq:norm_grad_success}
		\norm{\grad f(x_{k})} < \dfrac{3(1-\eta_1)\gamma^2}{L_H}.
	\end{equation}
	Then, the $k$th iteration is successful.
\end{proposition}
\begin{proof}
	First, we note that the condition for a successful step $\rho_k \geq \eta_1$ is equivalent to
	\begin{equation}
	\label{eq:alternativerhok}
		f(R_k(s_k)) -m_k(s_k) + (1-\eta_1)(m_k(s_k) - m_k(0)) \le 0.
	\end{equation}
	The proof then consists in finding an upper bound for the left-hand side that is negative. From Lemma~\ref{lemma:convex_decrease}, we have that
 	\begin{align*}
  		(1-\eta_1) \left( m_k(s_k) - m_k(0) \right) \leq (1-\eta_1)(-\dfrac{\gamma}{2}\norm{s_k}^2).
 	\end{align*}
 	Combining this property with~\eqref{eq:fkplus1-model} and~\eqref{eq:normskR3} gives
 	\begin{align*}
 		f(R_k(s_k)) -m_k(s_k) + (1-\eta_1)(m_k(s_k) - m_k(0))
 		&\leq \dfrac{L_H}{6}\norm{s_k}^3 - \dfrac{\gamma}{2}(1-\eta_1)\norm{s_k}^2\\
 		&= \norm{s_k}^2 \left(\dfrac{L_H}{6}\norm{s_k}-\dfrac{\gamma}{2}(1-\eta_1)\right) \\
 		&\le \norm{s_k}^2\left(\dfrac{L_H}{6}\dfrac{\norm{g_k}}{\gamma}-\dfrac{\gamma}{2}(1-\eta_1)\right).
 	\end{align*}
 	The right-hand side is negative by~\eqref{eq:norm_grad_success}, from which we conclude that~\eqref{eq:alternativerhok} holds, and the iteration is successful.
\end{proof}

We now show that if the gradient norm is small enough, the Newton step decreases the gradient norm.

\begin{proposition}%[Decrease in gradient norm]
\label{prop:decrease_grad_norm}
	Under~\aref{assu:second_order_retraction} and~\aref{assu:ss_R3}, let $x_k\in \Rgamma$ be
	an iterate produced by Algorithm~\ref{algo:exact_RTR} such that $s_k$ is the
	Newton step, i.e., $H_k s_k=-g_k$, with $\norm{s_k}\leq \smallstep$ where $\smallstep$ is the constant
	from Lemma~\ref{lemma:bound_grad_pullback}. If
	\begin{align}\label{eq:condition_decrease_grad}
		\norm{\grad f(x_k)} <  \dfrac{2\gamma^2}{\cR \hat{L}_H},
	\end{align}
	and the iteration is successful, then $\norm{\grad f(x_{k+1})}_{x_{k+1}} < \norm{\grad f(x_{k})}_{x_{k}}$.
\end{proposition}
\begin{proof}
	Combining~\eqref{eq:grad_next_step},~\eqref{eq:normskR3} and~\eqref{eq:condition_decrease_grad} gives
	\begin{align*}
		\norm{g_{k+1}} &\leq \dfrac{\cR \hat{L}_H}{2} \norm{s_k}^2
		\leq \dfrac{\cR \hat{L}_H}{2} \dfrac{\norm{g_k}^2}{\gamma^2}
		\leq \dfrac{\cR \hat{L}_H}{2\gamma^2} \norm{g_k} \cdot \norm{g_k}
		< \norm{g_k}. \qedhere
	\end{align*}
\end{proof}
In order to derive a local convergence result, we show
that, if $x_k \in \Rgamma$ and the gradient norm is small enough, the Newton step remains in the neighborhood of the same minimizer. 
\begin{proposition}%[Neighbourhood of $x^*$]
\label{prop:stay_in_C3}
	Under~\aref{assu:second_order_retraction} and~\aref{assu:ss_R3}, let $x_k\in \Rgamma$ be an iterate produced by Algorithm~\ref{algo:exact_RTR} such that $s_k$ is the
	Newton step, and $\norm{s_k}\leq \smallstep$ where $\smallstep$ is the constant
	from Lemma~\ref{lemma:bound_grad_pullback}, and that
	\begin{equation}
	\label{eq:norm_grad2}
		\norm{\grad f(x_k)} < \min\left(\cns \gamma,\dfrac{\gamma\delta}{2\cdxRs},\dfrac{\gamma\delta}{2},
		\dfrac{2\gamma^2}{\cR \hat{L}_H}\right).
	\end{equation}
	Let $x^* \in \M$ be a local minimum of problem~\eqref{eq:P}
	such that $\dist \left( x_k, x^*\right) \le \delta$ and $f$ is geodesically $\gamma$-strongly convex on
	$\{y \in \M: \dist (y,x^*) < 2 \delta\}$. If the iteration is successful, then
	$\dist\left(x_{k+1},x^*\right) < \delta$.
\end{proposition}
\begin{proof}
	We first show that $\dist\left(x_{k+1},x^*\right) < 2\delta$. Using~\eqref{eq:normskR3} and~\eqref{eq:norm_grad2}, we have that
	\begin{align*}
		\norm{s_k} &\leq \dfrac{\norm{g_k}}{\gamma} <\dfrac{\cns \gamma}{\gamma} = \cns.
	\end{align*}
	It follows from Lemma~\ref{lemma:dist-retraction} that
	\[
		\dist(x_{k+1},x_k) \le \cdxRs\norm{s_k}
		\leq \cdxRs\dfrac{\norm{g_k}}{\gamma}
		\leq \dfrac{ \cdxRs}{\gamma}	\dfrac{\gamma\delta}{2\cdxRs}
		=\dfrac{\delta}{2},
	\]
	where the last inequality is due to~\eqref{eq:norm_grad2}. As a result,
	\begin{align*}
		\dist(x_{k+1},x^*) \le  \dist(x_{k+1},x_k) + \dist(x_k,x^*) \le \dfrac{\delta}{2} + \delta < 2\delta.
	\end{align*}
	By definition of $x^*$, we know that $f$ is geodesically $\gamma$-strongly convex over $S\subseteq \M$, a
	subset of $\M$ that includes $x^*$ and $x_{k+1}$. Consider a geodesic $c\colon [0,1] \to S$ contained in $S$ with $c(0) = x_{k+1}$ and $c(1) = x^*$, such that
	$\dist(x_{k+1},x^*) \leq L(c)$ where $L(c)=\norm{c'(0)}_{c(0)}$ is the length of the geodesic path. From~\citep[Theorem 11.21]{boumal2023}, we have
	\begin{equation*}
		f(x^*) \geq f(x_{k+1}) + \inner{\grad f(x_{k+1})}{c'(0)}_{x_{k+1}} + \dfrac{\gamma}{2}L(c)^2.
	\end{equation*}
	Using $f(x^*) \leq f(x_{k+1})$ gives
	\begin{align*}
		\dfrac{\gamma}{2}L(c)^2 &\leq \inner{\grad f(x_{k+1})}{-c'(0)}_{x_{k+1}}		&\leq \norm{\grad f(x_{k+1})}_{x_{k+1}} \norm{c'(0)}_{x_{k+1}}		&= \norm{\grad f(x_{k+1})}_{x_{k+1}} L(c).
	\end{align*}
	Therefore, we have $L(c) \leq \dfrac{2}{\gamma}\norm{\grad f(x_{k+1})}_{x_{k+1}}$. To conclude, recall
	that we have $\norm{s_k} \le \smallstep$ by assumption, thus Proposition~\ref{prop:decrease_grad_norm}
	applies, and we obtain
	\begin{align*}
		\dist(x_{k+1},x^*)
		&\leq L(c) \leq \dfrac{2}{\gamma}\norm{\grad f(x_{k+1})}_{x_{k+1}}
		< \dfrac{2}{\gamma}\norm{\grad f(x_{k})}_{x_{k}}
		< \dfrac{2}{\gamma}\dfrac{\gamma\delta}{2} = \delta.\qedhere
	\end{align*}
\end{proof}

We now characterize the local convergence of Algorithm~\ref{algo:exact_RTR}. The quadratic convergence rate stems from the following proposition.

\begin{proposition}\label{prop:newton_step}
	Under~\aref{assu:second_order_retraction} and~\aref{assu:ss_R3}, let $x_k\in \Rgamma$ be an iterate produced by Algorithm~\ref{algo:exact_RTR} such that $s_k$ is a
	Newton step, and $\norm{s_k}\leq \smallstep$ where $\smallstep$ is the constant
	from Lemma~\ref{lemma:bound_grad_pullback}. Suppose further that the $k$th iteration is
	successful. Then,
	\begin{equation*}
		\dfrac{\cR \hat{L}_H}{2\gamma^2}\norm{\grad f(x_{k+1})}_{x_{k+1}}
		\leq \left( \dfrac{\cR \hat{L}_H}{2\gamma^2}\norm{\grad f(x_{k})}_{x_{k}}\right)^2.
	\end{equation*}
\end{proposition}
\begin{proof}
	Using Lemma~\ref{lemma:grad_next_step} and Lemma~\ref{lemma:normskR3} gives
	\begin{align*}
		\norm{\grad f(x_{k+1})}_{x_{k+1}}
		\le \dfrac{\cR \hat{L}_H}{2}\norm{s_k }_{x_k}^2
		\le \dfrac{\cR \hat{L}_H}{2\gamma^2}\norm{\grad f(x_k) }_{x_k}^2,
	\end{align*}
	and multiplying both sides by $\dfrac{\cR \hat{L}_H}{2\gamma^2}$ yields the
	desired conclusion.
\end{proof}

Proposition~\ref{prop:newton_step} guarantees that the gradient norm decreases quadratically along Newton steps, provided that the gradient norm is small enough. In Proposition~\ref{prop:tr_local_convergence_Cora}, we show that the local phase of Algorithm~\ref{algo:exact_RTR} begins once the trust-region method generates a point in $\Rgamma$ with a small enough gradient.
\begin{proposition}[Local convergence of Algorithm~\ref{algo:exact_RTR}]
\label{prop:tr_local_convergence_Cora}
	Suppose that~\aref{assu:second_order_retraction},~\aref{assu:ss_R3} and~\aref{assu:hessian_lipschitz} hold.
	Let $x_k\in \Rgamma$ be an iterate produced by Algorithm~\ref{algo:exact_RTR}, and let $x^* \in \M$ be
	a local minimum of~\eqref{eq:P} such that $\dist(x_k,x^*)\leq \delta$ and $f$ is geodesically
	$\gamma$-strongly convex on $\{y\in \M: \dist(y,x^*) < 2\delta\}$. Finally, suppose that 
		\begin{equation}
	\label{eq:gradcondlocalcv}
		\norm{\grad f(x_k)} < \min\left(\cquad \min\left(\gamma,\gamma^2,\gamma\delta\right),\gamma\Delta_k\right)
	\end{equation}
	where $\cquad=\min\!\left( \dfrac{3(1-\eta_1)}{L_H},\smallstep,\dfrac{1}{\cdxRs},\dfrac{1}{2},\dfrac{1}{\cR\,\hat{L}_H}\right)$.
	Then, all subsequent iterations are successful Newton steps that remain in $\Rgamma$. Moreover, the sequence of gradient norms
	$\left(\norm{\grad f(x_k)}\right)_k$ converges quadratically to zero.
\end{proposition}
\begin{proof}
 Lemma~\ref{lemma:normskR3} and~\eqref{eq:gradcondlocalcv} give $\norm{s_k} \leq \norm{ \grad f(x_k)}/\gamma < \Delta_k$, and
	thus $s_k$ is the Newton step.
	In addition, the condition~\eqref{eq:gradcondlocalcv} also implies~\eqref{eq:norm_grad_success}, thus
	the $k$th iteration is successful by Proposition~\ref{prop:very_successful_steps}. Similarly,~\eqref{eq:gradcondlocalcv}	implies~\eqref{eq:condition_decrease_grad},~\eqref{eq:norm_grad2}, and $\norm{s_k}\leq \smallstep$, which yields
	$\norm{g_{k+1}} <\norm{g_k}$ (Proposition~\ref{prop:decrease_grad_norm}) and $\dist(x_{k+1},x^*)< \delta$
	(Proposition~\ref{prop:stay_in_C3}).

	Since $\norm{g_{k+1}}< \norm{g_k} < \gamma \Delta_k \leq \gamma \Delta\kplus$, the same reasoning applies
	at iteration $k+1$, and at every subsequent iteration by induction, proving the first part of the
	result. Finally, quadratic convergence follows from repeated application of Proposition~\ref{prop:newton_step}, as~\eqref{eq:gradcondlocalcv} implies
	$\norm{\grad f(x_k)} < \gamma^2/(\cR \hat{L}_H)$.
	Indeed, for any index $l \ge k$, we have
	\begin{equation}\label{eq:proof_quadratic}
			\dfrac{\hat{L}_H\cR}{2\gamma^2}\norm{\grad f(x_{l})} \le
		\left( \dfrac{\hat{L}_H\cR}{2\gamma^2}\norm{\grad f(x_{k})}\right)^{2^{l-k}}
		\le \left( \dfrac{1}{2}\right)^{2^{l-k}},
	\end{equation}
	which characterizes quadratic convergence.
\end{proof}

We emphasize that the local convergence property is an integral part of our \emph{global}
convergence analysis. Deriving global rates of convergence (i.e., complexity results) is the subject
of the next section.

%%%%%%%%%%%%%%%%%%%%%%%%%%%%%%%%%%%%%%%%%%%%%%%%%%%%%%%%%%%%%%%%%%%%%%%%%%%%%%%%%%%%%%%%%%%%%%%%%%%%%%%%%%%%%
\subsection{Complexity bounds}
\label{ssec:wccexact}

In this section, we combine the results from Sections~\ref{ssec:lemmasexact} and~\ref{ssec:localcvexact} to
obtain complexity bounds. More precisely, we seek a bound on the number of iterations
Algorithm~\ref{algo:exact_RTR} requires to reach an iterate $x_K \in \M$ that is an $(\varepsilong, \varepsilonH)$-second-order critical point---defined in Equation~\eqref{eq:target}.

Following Section~\ref{ssec:localcvexact}, we can bound the number of iterations in the local phase necessary
to satisfy~\eqref{eq:target}. The result below is a direct corollary of
Proposition~\ref{prop:tr_local_convergence_Cora}.

\begin{theorem}
\label{theorem:local_phase}
	Let the assumptions of Proposition~\ref{prop:tr_local_convergence_Cora} hold for $x_k \in \Rgamma$
	generated by Algorithm~\ref{algo:exact_RTR}. Then, the algorithm returns an iterate satisfying
	\eqref{eq:target} in at most
	\begin{equation}
	\label{eq:local_phase}
		\log_2 \log_2 \left( \dfrac{2\gamma^2}{\cR \hat{L}_H \varepsilon_g}\right)
	\end{equation}
	iterations following iteration $k$.
\end{theorem}
\begin{proof}
	For any $l \ge k$, it follows from Equation~\eqref{eq:proof_quadratic} that if $\norm{\grad f(x_l)} \geq \varepsilon_g$, it must be that
	\begin{equation*}
		l-k \leq \log_2 \log_2 \left( \dfrac{2\gamma^2}{\hat{L}_H \cR  \varepsilon_g}\right).\qedhere
	\end{equation*}
\end{proof}

Our main complexity result comes under the assumption that $f$ is lower bounded on $\M$.
\begin{assumption}
\label{assu:lower_bound}
	There exists $f^*>-\infty$ such that $f(x) \geq f^*$ for all $x\in \M$.
\end{assumption}

We first give an upper bound on the number of successful steps for Algorithm~\ref{algo:exact_RTR}.
\begin{theorem}[Number of successful iterations of Algorithm~\ref{algo:exact_RTR}]
\label{thm:successful_complexity_exact}
	Suppose that~\aref{assu:second_order_retraction}--\aref{assu:lower_bound} hold. Algorithm~\ref{algo:exact_RTR}
	produces an iterate satisfying~\eqref{eq:target} in at most
	\begin{align*}
		\dfrac{C}{\min\!\left(\ualpha^2,\ualpha^{4/3}\ubeta,\ualpha^{4/3}\ugamma,\ualpha^{2/3}\ugamma^2,\ubeta^3,\ubeta^2\ugamma,	\ubeta\ugamma^2,
		\ugamma^3,\ugamma^2\udelta\right)}
		+ 1 + \log_2 \log_2 \left( \dfrac{2\gamma^2}{\cR \hat{L}_H \varepsilon_g}\right)
	\end{align*}
	successful iterations, where the constant $C>0$ depends on $\cH,\cDelta,\Delta_0,\smallstep,\cR,\hat{L}_H,\eta_1,\cquad$,
	and for any $\theta \in \{\alpha,\beta,\gamma,\delta\}$, we define
	$\underline{\theta}=\min(1,\theta)$.
\end{theorem}
\begin{proof}
	Let $K \in \mathbb{N}$ such that Algorithm~\ref{algo:exact_RTR} has not produced an
	iterate satisfying~\eqref{eq:target} by iteration $K$.
	Let $\calS=\{k \le K: \rho_k \ge \eta_1\}$ denote the set of indices corresponding to successful
	(and very successful) iterations up to index $K$. We partition the set of iterations
	as follows:
	\begin{align*}
		\Ialpha &= \{k\in \calS: x_k\in \Ralpha\}\\
		\Ibeta &= \{k\in \calS: x_k\in \Rbeta\setminus \Ralpha\}\\
		\Igamma &= \{k\in \calS: x_k\in \Rgamma \setminus \Ralpha\}.
	\end{align*}
	We now bound the decrease in function value for all three sets.

	Let $k \in \Ialpha$, using Lemma~\ref{lemma:cauchy_decrease} and
	Lemma~\ref{lemma:lower_bound_radius}, we obtain
	\begin{small}
	\begin{equation}
	\label{eq:decsuccalpha}
		f(x_k)-f(x_{k+1})
		\ge \eta_1 \left(m_k(0)-m_k(s_k)\right)
		\ge \dfrac{\eta_1}{2}  \min \left( \dfrac{\alpha}{\cH}, \Delta_k\right)\alpha
		\ge \dfrac{\eta_1}{2} \min \left( \dfrac{\alpha}{\cH},\Deltamin \right)\alpha.
	\end{equation}
	\end{small}

	For $k \in \Ibeta$, combining Lemma~\ref{lemma:eigenstep_decrease} with
	Lemma~\ref{lemma:lower_bound_radius} gives
	\begin{equation}
	\label{eq:decsuccbeta}
		f(x_k)-f(x_{k+1})
		\ge \eta_1 \left(m_k(0)-m_k(s_k)\right)
		\ge \dfrac{\eta_1}{2} \Delta_k^2\beta
		\ge \dfrac{\eta_1}{2} \Deltamin^2 \beta.
	\end{equation}

	We partition $\Igamma$ into iterations with long steps, short steps on the boundary of the trust region, and short steps inside the trust region: $\Igamma := \Igamma^l \cup \Igamma^{s,b} \cup \Igamma^{s,i}$, where
	\begin{align*}
		\Igamma^l &= \{k\in \Igamma: \norm{s_k} > \smallstep\}\\
		\Igamma^{s,b} &= \{k\in \Igamma: \norm{s_k} \le \smallstep, \norm{s_k} = \Delta_k \}\\
		\Igamma^{s,i} &= \{k\in \Igamma: \norm{s_k} \le \smallstep, \norm{s_k} < \Delta_k \}
	\end{align*}
	where $\smallstep >0$ is defined in Lemma~\ref{lemma:bound_grad_pullback}.

	If $k \in \Igamma^l$, Lemma~\ref{lemma:convex_decrease} yields
	\begin{equation}
	\label{eq:decsuccgammaL}
		f(x_{k+1})-f(x_k)
		\ge \dfrac{\eta_1}{2} \gamma\norm{s_k}^2 
		\ge \dfrac{\eta_1}{2} \gamma \smallstep^2.
	\end{equation}

	If $k \in \Igamma^{s,b}$, we use Lemma~\ref{lemma:convex_decrease} together with
	Lemma~\ref{lemma:lower_bound_radius} to obtain
	\begin{equation}
	\label{eq:decsuccgammaSB}
		f(x_{k+1})-f(x_k)
		\ge \dfrac{\eta_1}{2} \gamma \norm{s_k}^2 
		= \dfrac{\eta_1}{2} \gamma\Delta_k^2 
		\ge \dfrac{\eta_1}{2} \gamma \Deltamin^2.
	\end{equation}

	Finally, if $k \in \Igamma^{s,i}$, we partition further $\Igamma^{s,i}$ into
	$\Igamma^{s,i,s} \cup \Igamma^{s,i,l}$, where
	\begin{align*}
		\Igamma^{s,i,l} &= \left\lbrace k\in \Igamma^{s,i}:
		\norm{g_{k+1}} \ge \min\!\left(\cquad\min\!\left(\gamma,\gamma^2,\gamma\delta\right)\!,\gamma\Delta_k \right)\!\right\rbrace, \\
		\Igamma^{s,i,s} &= \Igamma^{s,i} \setminus \Igamma^{s,i,l}.
	\end{align*}

	If $k \in \Igamma^{s,i,l}$, Proposition~\ref{lemma:decrease_Newton} implies
	\begin{eqnarray}
	\label{eq:decsuccgammaSIL}
		f(x_k) - f(x\kplus)
		\geq \eta_1\gamma \frac{\norm{g\kplus}}{\hat{L}_H\,\cR}
 		&\geq &\dfrac{\eta_1}{\hat{L}_H\,\cR} \min\!\left(\cquad \min\!\left(\gamma^2,\gamma^3,\gamma^2\delta\right)\!,\gamma^2\Delta_k\right)
 		\nonumber \\
 		&\geq &\dfrac{\eta_1}{\hat{L}_H\,\cR} \min\!\left(\cquad \min\!\left(\gamma^2,\gamma^3,\gamma^2\delta\right)\!,\gamma^2\Deltamin\right).
% 		\min\{\cquad\,\min\{\gamma^2,\gamma^2\delta,\gamma^3\},\gamma^2\Deltamin\}.
	\end{eqnarray}

	Finally, if $k \in \Igamma^{s,i,s}$, either $x_{k+1} \in \Igamma$ and the local convergence phase
	begins at $x_{k+1}$ according to Proposition~\ref{prop:tr_local_convergence_Cora}; which produces an iterate that satisfies~\eqref{eq:target} in a number of iterations given by~\eqref{eq:local_phase}. Otherwise, we must
	have $x_{k+1} \in \Ialpha \cup \Ibeta$, and as a result we have
	\begin{equation}
	\label{eq:boundgammaSIS}
		\left| \Igamma^{s,i,s}\right| \le |\Ialpha|+|\Ibeta|+1
		+\log_2 \log_2 \left( \dfrac{2\gamma^2}{\cR \hat{L}_H \varepsilon_g}\right).
	\end{equation}
	It thus suffices to bound the cardinality of $\calS_1$ and $\calS_2$ to bound $\left|\Igamma^{s,i,s}\right|$.

	Thanks to~\aref{assu:lower_bound}, we have
	\begin{eqnarray*}
		f(x_0)-f^*
		&\ge &f(x_0)-f(x_{K}) \\
		&\ge &\sum_{\substack{k \in \calS}} f(x_k)-f(x_{k+1}) \\
		&\ge &\sum_{\substack{k \in \Ialpha }} f(x_k)-f(x_{k+1})
		+\sum_{\substack{k \in \Ibeta }} f(x_k)-f(x_{k+1})
		+\sum_{\substack{k \in \Igamma^l }} f(x_k)-f(x_{k+1}) \\
		& &+\sum_{k \in \Igamma^{s,b}} f(x_k)-f(x_{k+1})
		+\sum_{k \in \Igamma^{s,i,l}} f(x_k)-f(x_{k+1}).
	\end{eqnarray*}
	Putting~\eqref{eq:decsuccalpha}, \eqref{eq:decsuccbeta}, \eqref{eq:decsuccgammaL},
	\eqref{eq:decsuccgammaSB} and \eqref{eq:decsuccgammaSIL} together, we obtain
	\begin{eqnarray*}
		f(x_0)-f^*
		&\ge &\left|\Ialpha\right| \dfrac{\eta_1}{2} \min \left( \alpha/\cH,\Deltamin \right)\alpha
		+\left|\Ibeta\right| \dfrac{\eta_1}{2} \Deltamin^2 \beta
		+\left|\Igamma^l\right| \dfrac{\eta_1}{2} \smallstep^2 \gamma \\
		& &+\left|\Igamma^{s,b}\right|\dfrac{\eta_1}{2} \Deltamin^2 \gamma
		+\left|\Igamma^{s,i,l}\right| \dfrac{\eta_1}{\hat{L}_H\,\cR} \min\!\left(\cquad\min\!\left(\gamma^2,\gamma^3,\gamma^2\delta\right)\!,\gamma^2\Deltamin\right).
	\end{eqnarray*}
	Since all quantities on the right-hand side are nonnegative, we can bound each cardinality
	independently as follows:
	\begin{eqnarray*}
		|\Ialpha| &\le &\frac{2(f(x_0)-f^*)}{\eta_1} \max\!\left(\cH\alpha^{-2},\Deltamin^{-1}\alpha^{-1}\right) \\
		|\Ibeta| &\le &\frac{2(f(x_0)-f^*)}{\eta_1} \Deltamin^{-2}\beta^{-1} \\
		|\Igamma^l| &\le &\frac{2(f(x_0)-f^*)}{\eta_1}\smallstep^{-2}\gamma^{-1} \\
		|\Igamma^{s,b}| &\le &\frac{2(f(x_0)-f^*)}{\eta_1}\Deltamin^{-2}\gamma^{-1} \\
		|\Igamma^{s,i,l}| &\le &\frac{\cR\hat{L}_H(f(x_0)-f^*)}{\eta_1}\max\!\left(\cquad^{-1}\max\!\left(\gamma^{-2},\gamma^{-3},
		\gamma^{-2}\delta^{-1}\right)\!,\gamma^{-2}\Deltamin^{-1}\right).
	\end{eqnarray*}
	Using that $\Deltamin =\cDelta \min\!\left(\Delta_0, \alpha^{1/2}, \alpha^{2/3}, \beta, \gamma\right)\! \ge \cDelta\min\!\left(\Delta_0,1\right)\!\min\!\left(\ualpha^{2/3},\ubeta,\ugamma\right)$ yields the following
	upper bounds
	\begin{eqnarray*}
		|\Ialpha| &\le &\frac{2(f(x_0)-f^*)}{\eta_1} \max\!\left(\cH,\cDelta^{-1}\Delta_0^{-1},\cDelta^{-1}\right)\!
		\max\!\left(\ualpha^{-2}, \ualpha^{-1}\ubeta^{-1},\ualpha^{-1}\ugamma^{-1}\right) \\
		|\Ibeta| &\le &\frac{2(f(x_0)-f^*)}{\eta_1} \max\!\left(\cDelta^{-2}\Delta_0^{-2},\cDelta^{-2}\right)\!
		\max\!\left(\ualpha^{-4/3}\ubeta^{-1},\ubeta^{-3},\ubeta^{-1}\ugamma^{-2}\right)	 \\
		|\Igamma^l| &\le &\frac{2(f(x_0)-f^*)}{\eta_1}\smallstep^{-2}\ugamma^{-1} \\
		|\Igamma^{s,b}| &\le &\frac{2(f(x_0)-f^*)}{\eta_1}\max\!\left(\cDelta^{-2}\Delta_0^{-2},\cDelta^{-2}\right)\!
		\max\!\left(\ualpha^{-4/3}\ugamma^{-1},\ubeta^{-2}\ugamma^{-1},\ugamma^{-3}\right) \\
		|\Igamma^{s,i,l}| &\le &\frac{\cR\hat{L}_H(f(x_0)-f^*)}{\eta_1}\max\!\left(\cquad^{-1},\cDelta^{-1}\Delta_0^{-1},\cDelta^{-1}\right)\!
		\max\!\left(\ualpha^{-2/3}\ugamma^{-2},\ubeta^{-1}\ugamma^{-2},\ugamma^{-3},\ugamma^{-2}\delta^{-1}\right)\!.
	\end{eqnarray*}
	Combining these bounds with~\eqref{eq:boundgammaSIS}, the total number of successful iterations
	is bounded as
	\begin{eqnarray*}
		|\calS| &= &|\Ialpha|+|\Ibeta|+|\Igamma^l|+|\Igamma^{s,b}|+|\Igamma^{s,i,l}|+|\Igamma^{s,i,s}| \\
		&\le &2|\Ialpha|+2|\Ibeta|+|\Igamma^l|+|\Igamma^{s,b}|+|\Igamma^{s,i,l}| + 1+\log_2 \log_2 \left( \dfrac{2\gamma^2}{\cR \hat{L}_H \varepsilon_g}\right).\\
		&\le &C\min\!\left(\ualpha^2,\ualpha^{4/3}\ubeta,\ualpha^{4/3}\ugamma,\ualpha^{2/3}\ugamma^2,\ubeta^3,\ubeta^2\ugamma,	\ubeta\ugamma^2,
		\ugamma^3,\ugamma^2\udelta\right)^{-1}
		+ 1+\log_2 \log_2 \left( \dfrac{2\gamma^2}{\cR \hat{L}_H \varepsilon_g}\right),
	\end{eqnarray*}
	where
	\begin{small}
	\begin{align*}
		C = \dfrac{(f(x_0)-f^*)}{\eta_1} \left[ 4\max\!\left(\cH,\cDelta^{-1}\Delta_0^{-1},\cDelta^{-1}\right)\!
		+6\max\!\left(\cDelta^{-2}\Delta_0^{-2},\cDelta^{-2}\right)	 	\right.\\
		\left.	+2\smallstep^{-2}
		+ \cR\hat{L}_H\max\!\left(\cquad^{-1},\cDelta^{-1}\Delta_0^{-1},\cDelta^{-1}\right)\!\right]. ~~~\qedhere
	\end{align*}
	\end{small}
%	By noticing that the bound on $|\calS|$ holds for any $K$ such that no iterate satisfying~\eqref{eq:target}
%	was computed prior to iteration $K$, we arrive at the desired conclusion.
\end{proof}

Theorem~\ref{thm:successful_complexity_exact} bounds the number of successful
iterations required to satisfy~\eqref{eq:target}, which corresponds to the number of gradient and Hessian evaluations. To account for the total number of iterations---the number of function evaluations, we
follow a common strategy and show that this number is at most a constant multiple of the number of successful
iterations.

\begin{lemma}%[Number of unsuccessful steps]
\label{le:unsuccessful_complexity_exact}
	Under the assumptions of Lemma~\ref{lemma:lower_bound_radius}, let $K \in \N$ and let $\calS$ denote
	the set of successful iterations of index $k \leq K$. Then,
	\begin{eqnarray*}
	\label{eq:succtotal}
		|\calS|
		&\ge
		&\frac{\log_{\tau_2}(1/\tau_1)}{1+\log_{\tau_2}(1/\tau_1)}(K+1) \nonumber \\
		&
		&- \frac{1}{1+\log_{\tau_2}(1/\tau_1)}
		\max\left[0,
		\log_{\tau_2}\left(\dfrac{1}{\cDelta}\right),
		\log_{\tau_2}\left(\dfrac{\Delta_0}{\cDelta\alpha^{\tfrac{1}{2}}}\right),
		\log_{\tau_2}\left(\dfrac{\Delta_0}{\cDelta\alpha^{\tfrac{2}{3}}}\right),
		\log_{\tau_2}\left(\dfrac{\Delta_0}{\cDelta \beta}\right),
		\right. \\
		&
		&\left.
		\log_{\tau_2}\left(\dfrac{\Delta_0}{\cDelta \gamma}\right)
		\right].
	\end{eqnarray*}
\end{lemma}
\begin{proof}
	The proof follows verbatim~\citep[Lemma 6.23]{boumal2023} with \eqref{eq:lower_bound_radius}
	replacing~\citep[Eq. (6.36)]{boumal2023} and $\tau_1,\tau_2$ replacing $\tfrac{1}{4}$ and $2$,
	respectively. Since $\cDelta \le 1$ by definition, the maximum is always a nonnegative quantity.
\end{proof}

Combining Theorem~\ref{thm:successful_complexity_exact} with Lemma~\ref{le:unsuccessful_complexity_exact}
gives the total iteration complexity.

\begin{theorem}[Iteration complexity of Algorithm~\ref{algo:exact_RTR}]\label{thm:complexity_exact}
	Under the assumptions of Theorem~\ref{thm:successful_complexity_exact},
	Algorithm~\ref{algo:exact_RTR} produces a point that satisfies~\eqref{eq:target} in at most
	\begin{small}
	\begin{eqnarray}
	\label{eq:complexity_exact}
		\dfrac{1+\log_{\tau_2}(1/\tau_1)}{\log_{\tau_2}(1/\tau_1)}\left[
		\dfrac{C}{\min\!\left(\ualpha^2,\ualpha^{4/3}\ubeta,\ualpha^{4/3}\ugamma,\ualpha^{2/3}\ugamma^2,\ubeta^3,\ubeta^2\ugamma,	\ubeta\ugamma^2,
		\ugamma^3,\ugamma^2\udelta\right)} + 1
		+ \log_2 \log_2 \left( \dfrac{2\gamma^2}{\cR \hat{L}_H \varepsilon_g}\right)\right]
		\nonumber\\ 
		\\
		+ \dfrac{1}{\log_{\tau_2}(1/\tau_1)}\max\left(
		\log_{\tau_2}\left(\dfrac{1}{\cDelta}\right),
		\log_{\tau_2}\left(\dfrac{\Delta_0}{\cDelta\alpha^{\tfrac{1}{2}}}\right),
		\log_{\tau_2}\left(\dfrac{\Delta_0}{\cDelta\alpha^{\tfrac{2}{3}}}\right),
		\log_{\tau_2}\left(\dfrac{\Delta_0}{\cDelta \beta}\right),
		\log_{\tau_2}\left(\dfrac{\Delta_0}{\cDelta \gamma}\right)
		\right) \nonumber
	\end{eqnarray}
	\end{small}
	iterations, where $C,\ualpha,\ubeta,\ugamma,\udelta$ are defined as in
	Theorem~\ref{thm:successful_complexity_exact}.
\end{theorem}

The bound of Theorem~\ref{thm:complexity_exact} holds for any values $\varepsilon_g>0$ and
$\varepsilon_H>0$, but it is especially relevant when $\varepsilon_g<\alpha$ and $\varepsilon_H<\beta$.
In that case, the iteration complexity~\eqref{eq:complexity_exact} is an improvement over the
$\calO\left(\max(\varepsilon_g^{-2}\varepsilonH^{-1}, \varepsilon_H^{-3})\right)$ bound of
Riemannian trust-region methods for generic nonconvex functions~\citep{boumal2019global}. Perhaps 
surprisingly, our bound does not depend on $\varepsilon_H$. This is because every iterate $x_k$ 
such that $\norm{\grad f(x_k)} \le \epsilon_g$ and $\lambda_{\mathrm{min}}(\Hess f(x_k))<-\varepsilonH$  belongs 
to $\Rbeta$, where the function decrease depends on $\beta$. In fact, if $\varepsilon_g<\alpha$ and $\varepsilon_H<\beta$, Algorithm~\ref{algo:exact_RTR} necessarily reaches some 
$x_k \in \Rgamma$ such that
\begin{equation*}
\label{eq:target_gamma}
\norm{\grad f(x_k)} \leq \varepsilon_g
\quad\text{and} \quad \lambda_{\mathrm{min}}(\Hess f(x_k)) \ge \gamma.
\end{equation*}
This ensures that the algorithm finds an approximate minimizer, which is not guaranteed in the general nonconvex case.  In that sense, the
strict saddle property allows to obtain stronger guarantees of optimality and improved complexity bounds.

On the other hand, if $\varepsilong \geq \alpha$ or $\varepsilonH\geq \beta$, an $(\varepsilong,\varepsilonH)$-critical point~\eqref{eq:target} need not belong to $\Rgamma$. The following table indicates the possible regions in which Algorithm~\ref{algo:exact_RTR} can terminate depending on the values of $\varepsilon_g$ and $\varepsilon_H$:
\begin{center}
\begin{tabu}{|c|c|c|}
\hline
Convergence of Algorithm~\ref{algo:exact_RTR} & $\alpha > \varepsilon_g$ & $\alpha \leq \varepsilon_g$ \\
\hline
$\beta > \varepsilon_H$ & $\Rgamma$  &  $\Ralpha \cup \Rgamma$  \\
\hline
$\beta\leq \varepsilon_H$ &  $\Rbeta \cup \Rgamma$ &  $\Ralpha \cup \Rbeta \cup \Rgamma$  \\
\hline
\end{tabu}
\end{center}

When the strict saddle parameters $\alpha$ and $\beta$ are known, one can always select
$\varepsilong$ and $\varepsilonH$ to ensure that the method reaches an iterate in $\Rgamma$.
In that case, the value of $\varepsilong$ only affect the complexity through
a logarithmic factor, and $\varepsilonH$ is irrelevant.

To end this section, we apply our complexity result to the examples of strict saddle functions from
Section~\ref{sec:background}.

\begin{example}
\label{ex:wccstrcvx}
As a continuation of Example~\ref{ex:strcvx}, let $f:\R^n \to \R$ be a $\gamma$-strongly convex function, choose $\alpha=1$ so that $f$ is $(1,1,\gamma,\tfrac{2}{\gamma})$-strict
	saddle, and let $\varepsilon_g \in (0,1)$. Then, by Theorem~\ref{thm:complexity_exact}, Algorithm~\ref{algo:exact_RTR}
	computes an iterate such that $\|\nabla f(x_k)\| \le \varepsilon_g$ in at most
	$\calO(\gamma^{-3}) +\log\log(\gamma^2\varepsilon_g^{-1})$ iterations. In comparison, a standard analysis
	of Newton's method with Armijo backtracking linesearch gives at most
	$\calO(\gamma^{-5}) + \log \log (\gamma^3 \varepsilon_g^{-1})$ iterations to find such a point~\citep{boyd2004}.
	Although our bound has a better dependency on $\gamma$, we believe that this is an artefact of
	the line-search analysis, which could be improved by changing the line-search condition.
\end{example}

\begin{example}
\label{ex:wccrayleigh}
As a continuation of Example~\ref{ex:rayleigh}, let 	$f: \mathbb{S}^{n-1} \to \R$ be defined by $f(x)=x^\T A x$, where $A \in \Rnn$ is a symmetric
	matrix with eigenvalues $\lambda_1 > \lambda_2 \ge \dots \ge \lambda_{n-1} > \lambda_n$.
	Then, by Theorem~\ref{thm:complexity_exact}, Algorithm~\ref{algo:exact_RTR}
	computes an iterate satisfying~\eqref{eq:target} in at most
	\begin{align*}
		\calO\left( \max\!\left(1,\dfrac{\lambda_{1}^{2}}{(\lambda_{n-1}-\lambda_n)^{2}}, \dfrac{\lambda_{1}^{4/3}}{(\lambda_{n-1}-\lambda_n)^{7/3}},\dfrac{\lambda_{1}^{2/3}}{(\lambda_{n-1}-\lambda_n)^{8/3}},\dfrac{1}{ (\lambda_{n-1}-\lambda_n)^{3}}, \dfrac{\lambda_{1}}{(\lambda_{n-1}-\lambda_n)^{3}}\right)\! \right)\\
		  + \calO\left( \log\log\left( (\lambda_{n-1}-\lambda_n)^{2}\varepsilon_g^{-1}\right) \right)
	\end{align*}
	iterations.
\end{example}

%%%%%%%%%%%%%%%%%%%%%%%%%%%%%%%%%%%%%%%%%%%%%%%%%%%%%%%%%%%%%%%%%%%%%%%%%%%%%%%%%%%%%%%%%%%%%%%%%%%
\section{Riemannian trust-region method with inexact subproblem minimization}
\label{sec:inexact}
%%%%%%%%%%%%%%%%%%%%%%%%%%%%%%%%%%%%%%%%%%%%%%%%%%%%%%%%%%%%%%%%%%%%%%%%%%%%%%%%%%%%%%%%%%%%%%%%%%%

In this section, we design an inexact variant of the Riemannian trust-region algorithm, that is tailored to
strict saddle functions. In each region of $\M$, some landscape-aware step is appropriate and ensures good convergence rates: gradient-like steps in $\Ralpha$, negative curvature steps in $\Rbeta$, and (regularized) Newton steps in $\Rgamma$. Our goal is to compute these steps approximately, without computing the entire spectrum of the Hessian to determine the region of the current point. The natural choice for this is the well-known truncated conjugate
gradient algorithm~\citep{toint1981towards,steihaug1983conjugate}, and its recent adaptations that yield second-order complexity guarantees~\citep{curtis2021trust}. We make minimial adjustements to the standard truncated conjugate gradient (tCG), in order to leverage the strict saddle structure and ensure convergence to second-order critical points. In a departure
from the exact setting, we explicitly use the strict saddle parameters $\alpha, \beta,\gamma$ in the inexact algorithm. This idea appears in a recent proposal
for nonconvex matrix factorization problems, which satisfy a different strict saddle 
property~\citep{oneill2023linesearch}.

%In this section, we design an inexact variant of the Riemannian trust-region algorithm tailored to
%strict saddle functions. Our approach builds on the well-known truncated conjugate
%gradient paradigm~\citep{toint1981towards,steihaug1983conjugate} as well as recent adaptations of
%this method that yield second-order complexity guarantees~\citep{curtis2021trust}. In a departure
%from the exact setting, we now require knowledge of the strict saddle parameters in order to identify
%which type of step is best suited for a given region. This paradigm resembles that of recent proposals
%for nonconvex matrix factorization problems using a different strict saddle 
%property~\citep{oneill2023linesearch}.

%%%%%%%%%%%%%%%%%%%%%%%%%%%%%%%%%%%%%%%%%%%%%%%%%%%%%%%%%%%%%%%%%%%%%%%%%%%%%%%%%%%%%%%%%%%%%%%%%%%
\subsection{Inexact algorithm and subroutines}
\label{ssec:inexalgos}

Recall that, at every iteration $k$, the trust-region subproblem is given by
\begin{equation}
\label{eq:inex_subpb}
	\min_{s\in \T_{x_k}\M} 
	m_k(s)	\text{ subject to } \norm{s}\leq \Delta_k,
\end{equation}
where $m_k$ is the quadratic model defined in~\eqref{eq:tr_model}. For nonzero $g_k$, we apply a truncated conjugate gradient method to find an approximate solution 
of the subproblem. Our variant of truncated conjugate gradient, described in Algorithm~\ref{algo:tCG}, 
is a Riemannian adaptation of~\citep{curtis2021trust}---a nonconvex trust-region method with 
complexity guarantees; which we further adapt to strict saddle problems.
\begin{algorithm}
\caption{Truncated Conjugate Gradient (tCG) for subproblem~\eqref{eq:inex_subpb}}
\label{algo:tCG}
\begin{algorithmic}
\State \textbf{Input:} Nonzero gradient $g_k$, Hessian matrix $H_k$, trust-region radius 
$\Delta_k$, accuracy parameter $\zeta \in (0,1)$, bound $\kappa_H \in [\|H_k\|,\infty)$, 
strict saddle parameter $\gamma>0$.
\State \textbf{Output:} trial step $s$ and flag \texttt{outCG} indicating termination type
\hrule
\vspace{1mm}
\State Define $\kmax=\min \left\lbrace n, \onehalf \sqrt{\dfrac{\lmax}{\gamma}} \ln \left(\dfrac{2\sqrt{\lmax}}{\zeta\sqrt{\gamma}} \max\left(\varepsilong^{-2}, \varepsilong\inv,\lmax/\gamma\right)\right)\right\rbrace$.
%\State Set $k_{max} \leftarrow	$ Equation~\eqref{eq:kmax} % \min\{n, \dfrac{1}{2}\sqrt{\kappa}\ln(4\kappa^{3/2}/\zeta)\}$ where $\kappa \leftarrow (M+2\gamma)/2$.
\State Set $y_0 =0$, $r_0 = g_k$, $p_0 = -g_k$, $j = 0$.
\While{$j<\kmax$}
\If{$\inner{y_j}{H_k  y_j} < \gamma \norm{y_j}^2$}
\State Set $d=\Delta_k y_j/\norm{y_j}$ and terminate with \texttt{outCG = not\_strongly\_convex}
\EndIf
\If{$\inner{p_j}{H_k p_j} < \gamma \norm{p_j}^2$}
\State Set $d=\Delta_k p_j/\norm{p_j}$ and terminate with \texttt{outCG = not\_strongly\_convex}
\EndIf
\State $\sigma_j \leftarrow \norm{r_j}^2/\inner{p_j}{H_k p_j}$ \Comment{Begin standard tCG procedure}
\State $y_{j+1} = y_j + \sigma_j p_j$
\If{$\norm{y_{j+1}} \geq \Delta_k$}
\State Compute $\bar{\sigma}_j\geq 0$ such that $\norm{y_j + \bar{\sigma}_j p_j}=\Delta_k$
\State \Return $s\leftarrow y_j + \bar{\sigma}_j p_j$ and \texttt{outCG = boundary\_step}
\EndIf
\State $r_{j+1} \leftarrow r_j + \sigma_j H_k  p_j$
\If{condition~\eqref{eq:tcg_small_residual} holds}
\State \Return $s\leftarrow y_{j+1}$ and \texttt{outCG = small\_residual}
\EndIf
\State $\tau_{j+1} \leftarrow \norm{r_{j+1}}^2/\norm{r_j}^2$
\State $p_{j+1} \leftarrow -r_{j+1} + \tau_{j+1}p_j$\Comment{end standard tCG procedure}
\State $j\leftarrow j+1$
\EndWhile
\State \Return $s\leftarrow y_{k_{max}}$ and \texttt{outCG = max\_iter}
\end{algorithmic}
\end{algorithm}

The main differences between Algorithm~\ref{algo:tCG} and~\citep[Algorithm 3.1]{curtis2021trust}
lie in the tolerance on the curvature and the stopping criterion. For the former, we use the strict
saddle constant $\gamma$ (the strong convexity constant in $\Rgamma$) instead of an arbitrary tolerance on the smallest eigenvalue of the Hessian. Throughout the 
iterations of Algorithm~\ref{algo:tCG}, we monitor the curvature of the Hessian along the directions generated by CG. Any curvature less than $\gamma$ indicates that the current iterate does not belong 
to $\Rgamma$, and triggers termination of tCG.
We also strengthen the stopping criterion of tCG from~\citep{curtis2021trust}, and use
\begin{equation}\label{eq:tcg_small_residual}
	\norm{r_{j+1}}
	\le 
	\zeta \min\!\left(\norm{g_k}^{2},
	\norm{g_k},
	\gamma \norm{y_{j+1}}\right),
\end{equation}
where  $r_{j} = \nabla m_k(y_j)$ is the residual of the CG algorithm
after $j$ iterations, and $\zeta \in (0,1)$. The term $\norm{g_k}^2$ in~\eqref{eq:tcg_small_residual} ensures a local quadratic convergence, as we show in Section~\ref{ssec:wccinexact}.

The remaining terms on the right-hand side of~\eqref{eq:tcg_small_residual} are used 
to certify a good decrease when the current iterate belongs to $\Ralpha$ or $\Rgamma$. Importantly,
when $H_k$ is positive definite, the residual condition~\eqref{eq:tcg_small_residual} 
is satisfied in $\min\!\left(n,\calO\left(\gamma^{-1/2}\right)\right)$ iterations. This property is intrinsic to the conjugate 
gradient algorithm.

\begin{lemma}%[Lemma 11 in \citet{royer2018complexity}]
\label{le:tcgits}
	Suppose that we apply the conjugate gradient algorithm to $g_k \in \T_{x_k}\M$ with 
	$\norm{g_k}>\varepsilon_g$ and 
	$H_k$ such that $ \gamma \Id \preceq H_k \preceq \lmax \Id$. Then, given a 
	tolerance $\zeta \in (0,1)$, CG finds a vector $y \in \T_{x_k}\M$ such that
	\begin{equation}\label{eq:quadratic_residual_tcg}
		\norm{H_k y + g_k} \leq  \zeta \min\!\left( \norm{g_k}^2, \norm{g_k}, \gamma \norm{y}\right)
	\end{equation}
	in at most
	\begin{equation}\label{eq:kmax}
 		\kmax := \min\!\left( n, 
 		\onehalf \sqrt{\dfrac{\lmax}{\gamma}} \ln \left(\dfrac{2\sqrt{\lmax}}{\zeta\sqrt{\gamma}} 
 		\max\left(\varepsilong^{-2},\varepsilong^{-1},\lmax/\gamma\right)\right)\right)
	\end{equation}
	iterations or, equivalently, Hessian-vector products.
\end{lemma}
\begin{proof}
	The result follows from~\citep[Lemma 11]{royer2018complexity}, and in particular the inequality
	\begin{equation*}
		\left\| H_k y_j + g_k \right\| \le 2 \sqrt{\dfrac{\lmax}{\gamma}} 
		\left(\frac{\sqrt{\lmax/\gamma}-1}{\sqrt{\lmax/\gamma}+1} \right)^j \norm{g_k}
	\end{equation*}
	that holds for any $j \le n$. The model is also minimized in at most $n$ steps, $\left\|H_k y_n + g_k\right\|=0$.
% with a minor adjustment to account for the squared gradient norm in~\eqref{eq:quadratic_residual_tcg}.
\end{proof} 
Lemma~\ref{le:tcgits} implies that one can use a cap on the number of iterations, 
along with checks on the curvature of $H_k$, to monitor the convergence of tCG. In particular, if $\lambda_{\min}(H_k)\geq \gamma$, then Algorithm~\ref{algo:tCG} satisfies~\eqref{eq:tcg_small_residual} in $\kmax$ iterations, which may be smaller than $n$ 
in large dimensions. When this is not the case, and that $x_k \notin \Ralpha$, the Hessian is guaranteed to have negative curvature, and we compute (an approximation of) the smallest eigenvalue. 

To this end, we rely on a minimum eigenvalue oracle (MEO), that either computes a direction of 
sufficient negative curvature, or certifies that such direction does not 
exist~\citep{royer2020newton,curtis2021trust}. One possible implementation of this procedure 
consists in constructing a full eigenvalue decomposition of the Hessian, which deterministically 
guarantees the desired outcome but requires access to the entire Hessian matrix (or, equivalently, 
$n$ Hessian-vector products). Cheaper variants rely on Krylov subspace methods, such as Lanczos' 
method, that provide the desired guarantee with high probability using 
potentially less than $n$ Hessian-vector products~\citep[Appendix B]{royer2020newton}. For the 
sake of generality, we describe the MEO as a probabilistic method in Algorithm~\ref{algo:MEO}. Similarly to Algorithm~\ref{algo:tCG}, a key 
difference with previous minimum eigenvalue oracles is that the strict saddle constant $\beta$ 
(associated with the region of negative curvature $\Rbeta$) replaces an a priori optimality 
tolerance on the minimum eigenvalue of the Hessian.

\begin{algorithm}
\caption{Minimum eigenvalue oracle (MEO)}\label{algo:MEO}
\begin{algorithmic}
	\State \textbf{Inputs:} Matrix $H_k$, trust-region radius $\Delta_k$, 
	failure probability tolerance $p\in (0,1)$, bound $\kappa_H \in [\norm{H},\infty)$, 
	strict saddle parameter $\beta$.
	\State \textbf{Outputs:} Either a certificate that $\lambda_{\min}(H_k) \ge -\beta$ valid with 
	probability at least $1-p$, or a vector $s \in \T_{x_k}\M$ such that
	\begin{equation}
	\label{eq:meostep}
		\inner{g_k}{s} \le 0, 
		\quad 
		\inner{s}{H_k s} \le -\dfrac{1}{2}\beta \norm{s}^2, 
		\quad \text{and} \quad \norm{s}= \Delta_k.
	\end{equation}
\end{algorithmic}
\end{algorithm}

Our inexact trust-region method (Algorithm~\ref{algo:inexact_RTR}) combines tCG (Algorithm~\ref{algo:tCG}) and the MEO (Algorithm~\ref{algo:MEO}). We first attempt to use tCG to solve 
subproblem~\eqref{eq:inex_subpb} approximately. If the current iterate has a large enough gradient ($x_k \in \Ralpha$) 
or tCG hits the boundary of the trust region, we use the step given by tCG. Otherwise, we call the MEO to estimate the minimum eigenvalue of the Hessian. If the MEO finds a 
direction of sufficient negative curvature, we use it as an approximate model minimizer. Otherwise, we proceed 
with the tCG step. 

\begin{algorithm}
\caption{Strict saddle RTR with inexact subproblem minimization}\label{algo:inexact_RTR}
\begin{algorithmic}[1]
	\State \textbf{Inputs:} Initial point $x_0\in \M$,
	initial and maximal trust-region radii $0<\Delta_0<\bar\Delta$,
	constants $0<\eta_1<\eta_2<1$ and $0<\tau_1<1<\tau_2$, failure probability tolerance 
	$p \in [0,1)$, strict saddle parameters
	$(\alpha,\beta,\gamma)$.% tolerances $0<\varepsilon_g \le \alpha$ and $0<\varepsilon_H$.
	\For{$ k = 1, 2, \dots $}
		\If{$\norm{g_k}>0$}
			\State Call Algorithm~\ref{algo:tCG} on the subproblem~\eqref{eq:inex_subpb}
			to obtain $s_k^{CG}$ and \texttt{outCG}.
		\EndIf
		\If{$\norm{g_k} \ge \alpha$ or \texttt{outCG=boundary\_step}}
			\State Set $s_k=s_k^{CG}$.
		\ElsIf{$\norm{g_k}=0$ or  (\texttt{outCG} $\in$ \texttt{\{max\_iter,not\_strongly\_convex,small\_residual\}} and $\norm{g_k}< \alpha$)}.
			\State Call Algorithm~\ref{algo:MEO} on the Hessian $H_k$.
			\If{Algorithm~\ref{algo:MEO} certifies that $\lambda_{\min}(H_k) > -\beta$}
				\State Set $s_k=s_k^{CG}$ if $\norm{g_k}>0$, otherwise terminate and return $x_k$.
			\Else 
				\State Set $s_k=s_k^{MEO}$, where $s_k^{MEO}$ is the output of Algorithm~\ref{algo:MEO}.
			\EndIf
		\EndIf
%		\If{ $\norm{g_k}=0$ or  (\texttt{outCG} $\in$ \texttt{\{max\_iter,not\_strongly\_convex,small\_residual\}} and $\norm{g_k}< \alpha$)}
%			\State Call Algorithm~\ref{algo:MEO} on the Hessian $H_k$.
%			\If{Algorithm~\ref{algo:MEO} certifies that $\lambda_{\min}(H_k) > -\beta$}
%				\State Set $s_k=s_k^{CG}$ if $\norm{g_k}>0$, otherwise terminate and return $x_k$.
%			\Else 
%				\State Set $s_k=s_k^{MEO}$, where $s_k^{MEO}$ is the output of Algorithm~\ref{algo:MEO}.
%			\EndIf
%		\Else \Comment{$\norm{g_k} \ge \alpha$ or \texttt{outCG=boundary\_step}}.
%			\State Set $s_k=s_k^{CG}$.
%		\EndIf
		\State Compute $\rho_k = \dfrac{f(x_k) - f\left(R_{x_k}(s_k)\right)}{m_k(0) - m_k(s_k)}$ and
		set $x_{k+1} =
				\left\lbrace
					\begin{aligned}
					& R_k(x_k) & \text{ if } \rho_k \ge \eta_1\\
					& x_k & \text{ otherwise.}
					\end{aligned}
				\right.$\\
		\State Set $\Delta_{k+1} =
			\left\lbrace
				\begin{aligned}
					&\min\left(\tau_2\Delta_k,\bar \Delta\right) & \text{ if }   \rho_k >  \eta_2
					&&\text{ [very successful] }\\
					& \Delta_k & \text{ if }   \eta_2 \ge \rho_k \ge  \eta_1
					&&\text{ [successful] }\\
					& \tau_1\Delta_k & \text{ otherwise.}
					&&\text{ [unsuccessful] }
				\end{aligned}
			\right.$
	\EndFor
\end{algorithmic}
\end{algorithm}

Note that when $g_k=0$, the trust-region method calls Algorithm~\ref{algo:MEO} directly, 
and terminates if it certifies that $\lambda_{\min}(H_k) \ge -\beta$, since this implies 
$x_k \in \Rgamma$ is a local minimum. This corner case is not central to our complexity 
analysis.

For the rest of Section~\ref{sec:inexact}, we consider the following assumption on the MEO (Algorithm~\ref{algo:inexact_RTR}).

\begin{assumption}
\label{assu:iter_meo}
	For any iteration of Algorithm~\ref{algo:inexact_RTR}, if Algorithm~\ref{algo:MEO} is called 
	during this iteration, it outputs the correct answer, i.e., a certificate that 
	$\lambda_{\min}(H_k)>-\beta$ 
	or a step of curvature less than $-\tfrac{\beta}{2}$, in at most 
	\begin{equation}\label{eq:iter_meo}
		\Nmeo:=\min\!\left(n, 1+ \left\lceil \dfrac{1}{2}\ln(2.75 n/p^2) 
		\sqrt{\dfrac{\kappa_H}{\beta}} \right\rceil \right)
	\end{equation}
	iterations, with probability at least $1-p$.
\end{assumption}
Computing a full eigenvalue decomposition for Algorithm~\ref{algo:MEO} ensures 
that~\aref{assu:iter_meo} holds for any $p\ge 0$, with $\Nmeo=n$. The bound~\eqref{eq:iter_meo} 
also applies to Krylov subspace methods, such as Lanczos' method with an initial vector uniformaly distributed on the unit sphere~\citep{royer2020newton}. 

%%%%%%%%%%%%%%%%%%%%%%%%%%%%%%%%%%%%%%%%%%%%%%%%%%%%%%%%%%%%%%%%%%%%%%%%%%%%%%%%%%%%%%%%%%%%%%%%%%%
\subsection{Properties of inexact steps}
\label{ssec:inexlemmas}

In this section, we provide several decrease lemmas for the steps produced by 
Algorithms~\ref{algo:tCG} and~\ref{algo:MEO}. Our proof technique follows earlier works on Euclidean trust-region methods for general nonconvex functions~\citep{curtis2021trust}.

We first consider iterations of Algorithm~\ref{algo:tCG} where $x_k\in \Ralpha$ (large gradient norm).

\begin{lemma}%[Region Ralpha]
\label{le:tcggradalpha}
	Under~\aref{assu:ss_R3} and~\aref{assu:bounded_hessian}, 
	consider the $k$th iteration of Algorithm~\ref{algo:inexact_RTR}. Suppose that 
	$\norm{g_k}\ge \alpha$, so that Algorithm~\ref{algo:tCG} is called. Suppose that it outputs 
	$s_k$ with \texttt{outCG $\neq$ boundary\_step}. Then, we have
	\begin{equation}
	\label{eq:tcggradalpha}
		m_k(0)-m_k(s_k) 
		\geq 
		\dfrac{1}{2}\min \left(\Delta_k, \dfrac{\alpha}{\cH}\right)\alpha.
	\end{equation}
\end{lemma}
\begin{proof}
	By assumption, the step generated by tCG is taken, and it guarantees at 
	least as much model decrease as the Cauchy 
	step $s_k^C$, defined in the proof of Lemma~\ref{lemma:cauchy_decrease}. Indeed, the Cauchy step 
	corresponds to the first iterate of tCG~\citep{boumal2023,conn2000trust}. 
	Applying Lemma~\ref{lemma:cauchy_decrease} gives
	\begin{equation*}
		m_k(0)-m_k(s_k) 
		\ge m_k(0)-m_k(s^C_k) 
		\ge \frac{1}{2}\min\left(\Delta_k,\dfrac{\alpha}{\cH}\right)\alpha.\qedhere
	\end{equation*}
\end{proof}
The following lemma gives a model decrease when $x_k\in \Rbeta$ and the MEO is called. 
\begin{lemma}%[Not convex tCG]
\label{le:meodecrease}
	Under~\aref{assu:second_order_retraction} and \aref{assu:ss_R3},
	consider the $k$th iteration of Algorithm~\ref{algo:inexact_RTR}. Suppose that 
	Algorithm~\ref{algo:MEO} is called with $x_k \in \Rbeta$. Then,the method outputs a vector 
	$s_k \in \T_{x_k}\M$ satisfying
	\begin{equation}
	\label{eq:meodecrease}
		m_k(0)-m_k(s_k) 
		\geq 
		\dfrac{1}{4} \beta \Delta_k^2 .
	\end{equation}
	with probability at least $1-p$.
\end{lemma}
\begin{proof}
	Since $x_k \in \Rbeta$, Algorithm~\ref{algo:MEO} is 
	called with $H_k$ such that $\lambda_{\min}(H_k) \le -\beta$. Algorithm~\ref{algo:inexact_RTR} 
	then outputs a step satisfying~\eqref{eq:meostep} with probability $1-p$. Reasoning as in the 
	proof of Lemma~\ref{lemma:eigenstep_decrease}, we bound the model decrease as 
	\begin{equation*}
		m_k(0)-m_k(s_k) \ge -\dfrac{1}{2}\inner{s_k}{H_k s_k} \ge \dfrac{\beta}{4}\|s_k\|^2 
		= \dfrac{\beta}{4}\Delta_k^2.\qedhere
	\end{equation*}	
\end{proof}
Note that Lemma~\ref{le:meodecrease} is a probabilistic result, because it depends on the minimum 
eigenvalue oracle, but this oracle is only called for small gradients.

We now consider an iteration at which tCG terminates with a small residual, which satisfies~\eqref{eq:tcg_small_residual}.

\begin{lemma}%[tCG small residual]
\label{le:tcg_small_residual}
	Under~\aref{assu:second_order_retraction} and~\aref{assu:ss_R3}, 
	consider the $k$th iteration of Algorithm~\ref{algo:inexact_RTR}. Suppose that 
	Algorithm~\ref{algo:tCG} is called and outputs $s_k$ with 
	\texttt{outCG=small\_residual}. Then,
	\begin{equation}
	\label{eq:decrease_small_residual_1}
		m_k(0)-m_k(s_k) \ge \dfrac{1}{4} \gamma\norm{s_k}^2 .
	\end{equation}
	In addition, if $\norm{s_k} \leq \smallstep$,  we have
	\begin{equation}
	\label{eq:decrease_small_residual_2}
		m_k(0)-m_k(s_k) 
%		\ge \dfrac{\gamma}{4} \norm{s_k}^2 
		\ge  \dfrac{1}{2 ( \cR^2  + 2 \hat{L}_H \cR)}  \min \!\left(\norm{\grad f\left(R_{x_k}(s_k)\right)}^2  \gamma^{-1}, \gamma^3 \right).
	\end{equation}
\end{lemma}
\begin{proof}
	By assumption, we have $s_k=y_{j+1}$, where $y_{j+1}$ is the first iterate of tCG 
	that satisfies the residual condition~\eqref{eq:tcg_small_residual}. As a result,
	\begin{align}
	\inner{s_k}{H_k s_k} 
	&=  \inner{y_j + \sigma_j p_j}{ H_k (y_j + \sigma_j p_j)} \nonumber\\
	&=  \inner{y_j}{ H_k y_j} + \sigma_j \inner{y_j}{ H_k p_j } + \sigma_j \inner{p_j}{ H_k y_j } 
	+ \sigma_j^2 \inner{p_j}{ H_k p_j }\nonumber\\
	&= \inner{y_j}{H_k y_j } + \sigma_j^2\inner{p_j}{ H_k p_j},\label{eq:quadratic_sk}
	\end{align}
	where the last line follows from the standard properties of CG iterates 
	$y_j = \sum_{i=0}^{j-1} \sigma_i p_i$ and $\inner{p_j}{H_k p_i} =0$ for $i\neq j$, which implies 
	$\inner{p_j}{H_k y_j} =0$ (see, e.g., \cite[Chapter 5]{nocedalNumericalOptimization2006} or 
	\cite[Appendix A]{royer2020newton}). 
	In addition, Algorithm~\ref{algo:tCG} ensures that 
	$\inner{p_j}{H_k p_j}\geq \gamma \norm{p_j}^2$ and $\inner{y_j}{H_k y_j}\geq \gamma \norm{y_j}^2$, 
	hence
	\begin{equation}
	\label{eq:sk_gamma2_convex}
		\inner{s_k}{H_k s_k} 
		\ge \gamma \norm{y_j}^2 + \gamma \norm{\sigma_j p_j}^2
		\ge \dfrac{\gamma}{2}\norm{s_k}^2, 
	\end{equation}
	where the last inequality follows from $\norm{u}^2 + \norm{v}^2 \geq \onehalf \norm{u+v}^2$ for 
	all vectors $u,v\in \T_{x_k}\M$. Meanwhile, the model decrease satisfies
	\begin{eqnarray*}
		m_k(0)-m_k(s_k) 
		= -\inner{s_k}{g_k} - \dfrac{1}{2}\inner{s_k}{H_k s_k}
		&= &-\inner{s_k}{-H_k s_k + r_{j+1}} - \dfrac{1}{2}\inner{s_k}{H_k s_k} \\
		&= &\dfrac{1}{2}\inner{s_k}{H_k s_k} - \inner{r_{j+1}}{s_k} \\
		&= &\dfrac{1}{2}\inner{s_k}{H_k s_k},
	\end{eqnarray*}
	using that $r_{j+1}=H_k y_{j+1}+g_k=H_k s_k+g_k$ and $\inner{r_{j+1}}{y_{j+1}}=0$ by the orthogonality property of 
	the CG residual. Combining the last equality with~\eqref{eq:sk_gamma2_convex}, we obtain
	\begin{equation}
	\label{eq:s_k_decrease}
		m_k(0)-m_k(s_k)  \ge \dfrac{1}{4}\gamma \norm{s_k}^2,
	\end{equation}
	which proves the first part of the proposition.

Now assume $\norm{s_k}\leq \smallstep$ and let 
	$g_k^+=\grad f\left(R_{x_k}(s_k)\right)$. Using the proof of Lemma~\ref{lemma:grad_next_step} gives
	\begin{eqnarray}
	%\label{eq:grad_bound_quadratic}
		\norm{g_k^+} 
		&\le &\cR \norm{\nabla \hat f_k(s_k) - \grad f(x_k) + \grad f(x_k)} \nonumber\\
		&= &\cR\norm{\nabla \hat f_k(s_k) - \grad f(x_k) - H_k s_k +r_{j+1}} \nonumber \\
		&\le &\cR \norm{\nabla \hat f_k(s_k) - \grad f(x_k) - H_k s_k} + \cR \norm{r_{j+1}}\label{eq:grad_bound_quadratic_2}\\
		&\le &\dfrac{\hat{L}_H\cR}{2}\norm{s_k}^2 + \cR\gamma  \norm{s_k}, \nonumber
	\end{eqnarray}
	where the last line follows from the small residual condition~\eqref{eq:tcg_small_residual}.
	The inequality
	\begin{equation}
	\label{eq:quadnormsk}
		\dfrac{\hat{L}_H\cR}{2}\norm{s_k}^2 + \cR \gamma \norm{s_k} - \norm{g_k^+} \geq 0
	\end{equation}
	involves a univariate quadratic function of $\norm{s_k} \ge 0$, and thus~\eqref{eq:quadnormsk} 
	holds as long as 
	\begin{eqnarray*}
		\norm{s_k} 
		\ge \dfrac{- \cR \gamma +\sqrt{\cR^2\gamma^2 + 2\hat{L}_H\cR \norm{g_k^+}  }}{\hat{L}_H\cR } 
		&= &\left(\dfrac{- \cR +\sqrt{ \cR^2  + 2\hat{L}_H\cR \norm{g_k^+}\gamma^{-2}  }}{\hat{L}_H\cR }\right)\!\gamma\\
		&\ge &\left(\dfrac{- \cR + \sqrt{ \cR^2 +2\hat{L}_H \cR}}{\hat{L}_H \cR}\right) \min\!\left(\norm{g_k^+} \gamma^{-2}, 1\right)\gamma\\
		&= &\left(\dfrac{2}{ \cR + \sqrt{ \cR^2+2\hat{L}_H \cR}}\right) \min\!\left(\norm{g_k^+} \gamma ^{-1},\gamma\right) \\
		&\ge &\dfrac{1}{\sqrt{\cR^2+2\hat{L}_H \cR}}\min\!\left(\norm{g_k^+} \gamma ^{-1},\gamma\right)\!,
	\end{eqnarray*}
	where we used that $-a + \sqrt{a^2 + bt} \geq (-a + \sqrt{a^2 +b})\min(t,1)$ with $a= \cR$, $b=2\hat{L}_H \cR$ and 
	$t= \norm{g_{k+1}}\gamma^{-2}$ and that $\dfrac{2}{1+\sqrt{1+c}} \ge \dfrac{1}{\sqrt{1+c}}$ for any $c>0$. 
	Combining the above with~\eqref{eq:s_k_decrease}, we get
	\begin{eqnarray*}
		m_k(0) - m_k(s_k) 
		\geq \dfrac{1}{4}\gamma \norm{s_k}^2
		\geq \dfrac{1}{2 ( \cR^2  + 2\hat{L}_H \cR)}  \min\left(\norm{\grad f\left(R_{x_k}(s_k)\right)}^2 \gamma^{-1},\gamma^{3}\right),
	\end{eqnarray*}
	proving~\eqref{eq:decrease_small_residual_2}.
\end{proof}

The next lemma considers steps that lie on the boundary of the trust region. Note that our 
proof differs from the general nonconvex setting~\citep[Lemma 4.3]{curtis2021trust}, because we do not add an artificial regularizer in the subproblem.

\begin{lemma}%[tCG Boundary step]
\label{le:tcg_boundary}
	Under~\aref{assu:second_order_retraction} and~\aref{assu:ss_R3}, 
	consider the $k$th iteration of Algorithm~\ref{algo:inexact_RTR}. Suppose that 
	Algorithm~\ref{algo:tCG} is called and outputs $s_k$ together with 
	\texttt{outCG=boundary\_step}. Then,
	\begin{equation}
	\label{eq:inexact_boundary_decrease}
		m_k(0)-m_k(s_k) \geq \dfrac{1}{4}\gamma\Delta_k^2.
	\end{equation}
\end{lemma}
\begin{proof}
	Since \texttt{outCG=boundary\_step}, the step has the 
	form $s_k = y_j+\bar{\sigma}_j p_j$ with $\|s_k\| = \Delta_k$ and
	\begin{equation}
	\label{eq:step_alpha}
		0 \le \bar{\sigma}_j \le \sigma_j = \dfrac{\norm{r_j}^2}{\inner{p_j}{H_k p_j}} 
		= -\dfrac{\inner{g_k}{p_j}}{\inner{p_j}{H_k p_j}},
	\end{equation}
	where the last equality holds by definition of the CG residual.
	%$\norm{r_j}^2 = - \inner{r_j}{p_j} = - \inner{g_k}{p_j}$. 
	Since $\inner{p_j}{H_k p_j} \geq \gamma \norm{p_j}^2>0$, Equation~\eqref{eq:step_alpha} 
	implies
 	\begin{eqnarray*}
		-\bar{\sigma}_j\inner{g_k}{p_j} &\geq  \bar{\sigma}_j^2 \inner{p_j}{H_k p_j},
 	\end{eqnarray*}
	from which we obtain
  	\begin{eqnarray}
  	\label{eq:dec_boundary_pjaux}
		m_k(0)-m_k(\bar{\sigma}_j p_j) 
		&=  &-\bar{\sigma}_j\inner{g_k}{p_j} - \dfrac{\bar{\sigma}_j^2}{2} \inner{p_j}{H_k p_j} 
		\nonumber \\
		&\ge &\bar{\sigma}_j^2 \inner{p_j}{H_k p_j} - \dfrac{\bar{\sigma}_j^2}{2} \inner{p_j}{H_k p_j} 
		\nonumber \\
		&= &\dfrac{\bar{\sigma}_j^2}{2} \inner{p_j}{H_k p_j} \nonumber \\
		&\ge  &\dfrac{\gamma}{2} \norm{\bar{\sigma}_j p_j}^2.
	\end{eqnarray}
	The reasoning used to prove~\eqref{eq:decrease_small_residual_1} can be 
	applied to $m_k(0)-m_k(y_j)$ (recall that $\inner{y_j}{H_k y_j} \geq \gamma \norm{y_j}^2$), which gives 
	\begin{equation}
	\label{eq:dec_boundary_yjaux}
		m_k(0)-m_k(y_j) \geq \dfrac{\gamma}{2} \norm{y_j}^2. 
	\end{equation}
	Finally, using $m_k(s_k) - m_k(0) = m_k(y_j) - m_k(0) + m_k(\bar\sigma_j p_j) - m_k(0)$ (see~\eqref{eq:quadratic_sk}), we combine~\eqref{eq:dec_boundary_pjaux} and~\eqref{eq:dec_boundary_yjaux} to conclude as follows:
	\begin{eqnarray*}
		m_k(0) - m_k(s_k) 
		&= &m_k(0) - m_k(y_j) + m_k(0) - m_k(\bar{\sigma}_j p_j)\\
		&\ge &\dfrac{\gamma}{2} \norm{y_j}^2 + \dfrac{\gamma}{2} \norm{\bar{\sigma}_j p_j}^2\\
		&\ge &\dfrac{\gamma}{2} \left(\norm{y_j}^2 +\norm{\bar{\sigma}_j p_j}^2\right)\\
		&\ge &\dfrac{\gamma}{4} \norm{s_k}^2,
	\end{eqnarray*}
	where the last line holds because $\norm{u}^2 + \norm{v}^2 \geq \onehalf \norm{u+v}^2$ for all
	$u,v\in \T_{x_k}\M$.
\end{proof}

We end this section with a lower bound on the trust-region radius based on the decrease lemmas 
above, akin to the exact setting. 

\begin{lemma}%[Lower bound on TR radius]
\label{lemma:tr_radius_inexact}
	Let~\aref{assu:second_order_retraction},~\aref{assu:ss_R3},~\aref{assu:hessian_lipschitz} and 
	\aref{assu:bounded_hessian} hold. For any index $k\geq 0$, assume that all calls to the MEO (Algorithm~\ref{algo:MEO}) up to iteration $k$ with an iterate in $\Rbeta$ succeed in finding a direction of sufficient 
	negative curvature. Then, the trust-region radius $\Delta_k$ in 
	Algorithm~\ref{algo:inexact_RTR} satisfies
	\begin{equation}\label{eq:tr_radius_inexact}
		\Delta_k \geq \Deltaminin := \cDeltain \min\!\left(\Delta_0, \alpha^{1/2},
		\alpha^{2/3}, \beta, \gamma\right),
	\end{equation}
	where $\cDeltain = \min\!\left(1, 
	\tau_1\dfrac{3(1-\eta_1)}{2 L_H},
	\tau_1\sqrt{\dfrac{3(1-\eta_1)}{2 L_H}},
	\tau_1\sqrt[3]{\dfrac{3(1-\eta_1)}{2 \cH L_H}}\right)$.
\end{lemma}
\begin{proof}
	The proof follows the lines of the exact case (Lemma~\ref{lemma:lower_bound_radius}) by considering the 
	quantity $1-\rho_k$. However, rather than partitioning the iterates according to 
	the strict saddle regions, we consider the various steps that can be produced by the trust-region 
	algorithm.

	Consider first that tCG (Algorithm~\ref{algo:tCG}) is called and outputs $s_k$ with flag 
	\texttt{outCG=small\_residual}. Per Lemma~\ref{le:tcg_small_residual}, we know that 
	the model decrease satisfies~\eqref{eq:decrease_small_residual_1}. Using 
	\aref{assu:hessian_lipschitz}, we combine this with the bound~\eqref{eq:fkplus1-model} on $f(R_{x_k}(s_k))- m_k(s_k)$ to give
	\begin{eqnarray*}
		1-\rho_k 
		= \frac{f(R_{x_k}(s_k))- m_k(s_k)}{m_k(0)-m_k(s_k)} 
		\le \frac{(L_H/6)\norm{s_k}^3}{(\gamma/4)\norm{s_k}^2}
		= \frac{2\,L_H}{3\gamma}\norm{s_k} \le \frac{2\,L_H}{3\gamma}\Delta_k.
	\end{eqnarray*}
	Therefore, if \texttt{outCG=small\_residual} and $\Delta_k \le \frac{3(1-\eta_1)}{2\,L_H\gamma}$, then $\rho_k \ge \eta_1$ and the 
	iteration is successful.

	Suppose now that tCG outputs a boundary step, i.e., 
	\texttt{outCG=boundary\_step}. Lemma~\ref{le:tcg_boundary} combined with~\eqref{eq:fkplus1-model} gives
	\begin{eqnarray*}
		1-\rho_k 
		\le \frac{(L_H/6)\Delta_k^3}{(\gamma/4)\Delta_k^2} 
		\le \frac{2\,L_H}{3\gamma}\Delta_k,
	\end{eqnarray*}	
	hence the same conclusion than in the previous case holds.

	If the MEO is called with $x_k \in \Rbeta$, then by assumption  
	it succeeds in finding a direction of curvature at most $-\beta/2$. Lemma~\ref{le:meodecrease} 
	combined with~\eqref{eq:fkplus1-model} gives 
	\begin{eqnarray*}
		1-\rho_k 
		\le \frac{(L_H/6)\Delta_k^3}{(\beta/4) \Delta_k^2} 
		\le \frac{2 L_H}{3 \beta} \Delta_k.
	\end{eqnarray*}
	If $\norm{g_k}\geq \alpha$ and $s_k$ is not a boundary step, Lemma~\ref{le:tcggradalpha} combined with~\eqref{eq:fkplus1-model} gives
	\begin{eqnarray*}
		1-\rho_k 
		\le \frac{(L_H/6)\Delta_k^3}{\onehalf	\min \left(\Delta_k, \dfrac{\alpha}{\cH}\right)\alpha} 
		\le \frac{L_H}{3}\max\!\left(\frac{\Delta_k^2}{\alpha}, \frac{\cH \Delta_k^3}{\alpha^2} 
		\right).
	\end{eqnarray*}
%	It follows that the iteration is successful if
%	\begin{equation*}
%	%\label{eq:deltasmallinex}
%		\Delta_k \le \min\left\{
%		\frac{3(1-\eta_1)}{2 L_H}\beta,
%		\sqrt{\dfrac{3(1-\eta_1)}{2 L_H}}\alpha^{1/2},
%		\sqrt[3]{\dfrac{3(1-\eta_1)}{2 \cH L_H}}\alpha^{2/3}\right\}.
%	\end{equation*}
	We have thus established that the $k$th iteration is successful as long as
	\[
		\Delta_k \le \min\!\left(\dfrac{3(1-\eta_1)}{2 L_H}\gamma,
		\dfrac{3(1-\eta_1)}{2 L_H}\beta,
		\sqrt{\dfrac{3(1-\eta_1)}{ L_H}}\alpha^{1/2},
		\sqrt[3]{\dfrac{3(1-\eta_1)}{ \cH L_H}}\alpha^{2/3}\right)\!.
	\]
	holds, in which case $\Delta_{k+1} \ge \Delta_k$. Applying the updating rule on $\Delta_k$, we find 
	that the trust-region radius is lower bounded for any $k \geq 0$ by
	\begin{align*}
		\Delta_k 
		&\ge \min\!\left(\Delta_0,\tau_1\dfrac{3(1-\eta_1)}{2 L_H}\gamma,
		\tau_1\dfrac{3(1-\eta_1)}{2 L_H}\beta,
		\tau_1\sqrt{\dfrac{3(1-\eta_1)}{2 L_H}}\alpha^{1/2},
		\tau_1\sqrt[3]{\dfrac{3(1-\eta_1)}{2 \cH L_H}}\alpha^{2/3}\right) \\
		&\ge \cDeltain \min\!\left(\Delta_0,\alpha^{1/2},\alpha^{2/3}\,\beta,\gamma\right)\!.\qedhere
	\end{align*}
\end{proof}

The result of Lemma~\ref{lemma:tr_radius_inexact} relies on the MEO not failing 
to detect sufficient negative curvature if present. In the upcoming analysis, we 
bound the probability of such failure while deriving our main complexity result.

%%%%%%%%%%%%%%%%%%%%%%%%%%%%%%%%%%%%%%%%%%%%%%%%%%%%%%%%%%%%%%%%%%%%%%%%%%%%%%%%%%%%%%%%%%%%%%%%%%%
\subsection{Complexity bounds}
\label{ssec:wccinexact}

We now derive complexity bounds for our inexact trust-region method (Algorithm~\ref{algo:inexact_RTR}), and show that they depend logarithmically on the optimality tolerances, similarly to those established in 
Section~\ref{ssec:wccexact} for the exact algorithm.

We begin by describing the local convergence of Algorithm~\ref{algo:inexact_RTR}. The reasoning 
follows the exact setting detailed in Section~\ref{ssec:localcvexact}, and uses several results 
from that section.

\begin{theorem}
\label{theo:local_conv_inexact}
	Suppose that~\aref{assu:second_order_retraction} and~\aref{assu:ss_R3} hold.
	Let $x_k\in \Rgamma$ be an iterate produced by Algorithm~\ref{algo:inexact_RTR}, and let $x^* \in \M$ be
	a local minimum of~\eqref{eq:P} such that $\dist(x_k,x^*)\leq \delta$ and $f$ is geodesically
	$\gamma$-strongly convex on $\{y\in\M: \dist(y,x^*) < 2\delta\}$. Finally, suppose that either 
	$\grad f(x_k)=0$ or that
		\begin{align}\label{eq:condition_quadratic_conv_inexact}
		\norm{\grad f(x_k)} &< \min\!\left(	\cquadin \min\!\left(1,\gamma,\gamma^2,\gamma\delta\right),\gamma \Delta_k \right)\!,
	\end{align}
%	\begin{equation}\label{eq:condition_quadratic_conv_inexact}
%		\norm{\grad f(x_k)} < \min \left\{ 
%		\cquadin \min\{1,\gamma,\gamma^2,\gamma\delta\},
%		\gamma \Delta_k\right\},
%	\end{equation}
	where $\cquadin =\min\!\left(\dfrac{3(1-\eta_1)}{2 L_H},
	\smallstep,\dfrac{1}{\cdxRs},\dfrac{1}{2},\dfrac{1}{\cR(2+\lHhat)}
%	,\dfrac{1}{\cR\,\hat{L}_H}
	\right)\!$.
	Then, either Algorithm~\ref{algo:inexact_RTR} terminates or computes an iterate satisfying~\eqref{eq:target} 
	in at most 
	\begin{equation}\label{eq:log_log_inexact}
		\log_2 \log_2 \left( \dfrac{2\min (1,\gamma^2)}{\cR (\lHhat+2)\varepsilon_g}\right)
	\end{equation}
	iterations.
\end{theorem}
\begin{proof}
	Since $x_k \in \Rgamma$, we know that $\lambda_{\min}(H_k) \ge \gamma>0$. 
	If $\grad f(x_k)$ is zero, then Algorithm~\ref{algo:inexact_RTR} 
	calls Algorithm~\ref{algo:MEO}, which \emph{deterministically} outputs a certificate that 
	$\lambda_{\min}(H_k) > -\beta$ since no negative curvature direction exists, 
	hence the algorithm terminates.
	
	Consider now $\norm{\grad f(x_k)}>0$. Since $x_k \in \Rgamma$, Lemma~\ref{lemma:normskR3} and~\eqref{eq:condition_quadratic_conv_inexact} imply that the global 
	minimizer of the trust-region subproblem $s_k^{*}$ satisfies 
	$
		\norm{s_k^{*}} \le \norm{g_k}/\gamma < \Delta_k.
	$
 As a result, $s_k^*$ is the minimizer of the quadratic model, and lies inside the trust region. The iterates of CG have a norm not greater than the model minimizer, thus we have 
%	Given that Algorithm~\ref{algo:tCG} computes an approximate minimum of that model, we have
	\begin{equation}
	\label{eq:tcg_step_bounded_grad}
		\norm{s_k} \le \norm{s_k^*} < \Delta_k.
	\end{equation}
	We show that this step leads to a successful iteration. As in Proposition~\ref{prop:very_successful_steps}, we obtain from~\eqref{eq:condition_quadratic_conv_inexact} along with the 
	decrease~\eqref{eq:decrease_small_residual_1} and~\eqref{eq:fkplus1-model} that
	\begin{eqnarray*}
		f(R_{x_k}(s_k))-m_k(s_k) + (1-\eta_1)(m_k(s_k)-f(x_k)) 
		&\le 
		&\dfrac{L_H}{6}\norm{s_k}^3 - (1-\eta_1) \frac{\gamma}{4}\norm{s_k}^2 \\
		&\le 
		&\norm{s_k}^2 \left( \frac{L_H}{6}\frac{\|g_k\|}{\gamma}-\frac{\gamma}{4}(1-\eta_1) \right)
		< 0,
	\end{eqnarray*}
	showing that $\rho_k \ge \eta_1$ and thus the iteration is successful.
	
	Because $H_k \succeq \gamma \Id$ and the model minimizer is inside the trust-region, tCG terminates with a vector $s_k$ that satisfies the small residual condition~\eqref{eq:tcg_small_residual}. We combine 
	Lemma~\ref{lemma:hessian_lipschtiz_grad_ineq}, Lemma~\ref{lemma:bound_grad_pullback}, Lemma~\ref{lemma:normskR3} and~\eqref{eq:condition_quadratic_conv_inexact} to obtain
	\begin{eqnarray}
	\label{eq:gkp1inexact}
		\norm{g\kplus}_{x_{k+1}} \le \cR \norm{g_k}_{x_{k}}  
		&\le &\cR \norm{\nabla \hat f_k(s_k) - g_k - H_k s_k}_{x_{k}}  + \cR \norm{r_{j+1}}_{x_{k}}
		\nonumber \\
		&\le &\cR\dfrac{\lHhat}{2}\norm{s_k}^2_{x_{k}} +  \cR\norm{g_k}_{x_{k}}^{2} 
		\nonumber \\
		&\le &\cR\left(\dfrac{1}{\gamma^2}+\dfrac{\lHhat}{2}\right)\norm{g_k}_{x_{k}}^2 
		\nonumber \\
		&\le &\cR \dfrac{2+\lHhat}{2\min\!\left(1,\gamma^2\right)}\norm{g_k}_{x_{k}}^2 	
		< \frac{1}{2}\norm{g_k}_{x_{k}}.
	\end{eqnarray}
	
	As a result, we have $\norm{g_{k+1}} < \norm{g_k} <\gamma \Delta_k \leq \gamma \Delta\kplus$. 
	Applying Proposition~\ref{prop:stay_in_C3} ensures that
	$\dist(x\kplus,x^*) \le \delta$ (note that $\cquadin<\cquad$ and thus 
	\eqref{eq:norm_grad2} holds). We can apply the above reasoning to $x_{k+1}$, which 
	guarantees that either the method terminates or all subsequent iterates correspond to 
	successful iterations with truncated CG steps that satisfy~\eqref{eq:tcg_small_residual}. 
	Finally, to bound the number of iterations before reaching some $x_{\ell}$ such that 
	$\norm{\grad f(x_{\ell})} < \varepsilon_g$, we rewrite~\eqref{eq:gkp1inexact} as
	\begin{equation*}
		\cR\dfrac{(\lHhat+2)}{2\min\!\left(1,\gamma^2\right)}\norm{g_{k+1}} 
		\le 
		\left(\cR\dfrac{(\lHhat+2)}{2\min\!\left(1,\gamma^2\right)}\norm{g_k}\right)^2. 
	\end{equation*}
	This shows the iteration bound~\eqref{eq:log_log_inexact} for the local phase, using the proof of Theorem~\ref{theorem:local_phase}. 
%	Akin to the proof of Theorem~\ref{theorem:local_phase}, we can then show that the number of 
%	iterations before satisfying $\norm{g_{\ell}}<\varepsilon_g$ is upper bounded 
%	by~\eqref{eq:log_log_inexact}.
\end{proof}

Having characterized the local convergence behavior of our method, we can now derive 
global convergence results. As in the exact setting, we first bound the number of 
successful iterations.

\begin{theorem}[Number of successful iterations for Algorithm~\ref{algo:inexact_RTR}]
\label{thm:successful_complexity_inexact}
	Suppose that~\aref{assu:second_order_retraction}--\aref{assu:iter_meo} hold. 
	Algorithm~\ref{algo:inexact_RTR} either terminates or
	produces an iterate satisfying~\eqref{eq:target} in at most
	\begin{align*}
		K_{\calSin} 
		:= \dfrac{\Cinex}{\min\!\left(\ualpha^2,\ualpha^{4/3}\ubeta,\ualpha^{4/3}\ugamma,\ubeta^3,
		\ubeta^2\ugamma,\ubeta\,\ugamma^2,\ugamma^3,
		\ugamma\udelta^2 \right)}
		+ 1 + \log_2 \log_2 \left( \dfrac{2\ugamma^2}{\cR \hat{L}_H \varepsilon_g}\right)
	\end{align*}
	successful iterations with probability $(1-p)^{K_{\calSin}}$, where the constant $\Cinex>0$ depends on 
	$\cH$, $\cDeltain$, $\Delta_0$, $\smallstep$, $\cR$, $\hat{L}_H$, $\eta_1$, $\cquadin$, and 
	$\ualpha,\ubeta,\ugamma,\udelta$ are defined in Theorem~\ref{thm:successful_complexity_exact}.
\end{theorem}
\begin{proof}
	Let $K \in \N$ such that Algorithm~\ref{algo:inexact_RTR} has not produced an iterate 
	satisfying~\eqref{eq:target} by iteration $K$. Then the method cannot terminate before iteration 
	$K$, and tCG is called at every iteration. We partition the set $\calSin=\{k \le K: \rho_k \ge \eta_1\}$ of successful iterations based on properties of the steps, i.e., we define
	\begin{eqnarray*}
		\calSin_0 &= &\left\{k \in \calSin: 
		\mathrm{\texttt{outCG=boundary\_step}}\right\}, \\
		\calSin_1 &= &\left\{k \in \calSin \setminus \calSin_0: 
		x_k \in \Ralpha \right\}, \\
		\calSin_2 &= &\left\{k \in \calSin \setminus \calSin_0: 
		x_k \in \Rbeta \setminus \Ralpha \right\}, \\
		\calSin_3^{l} &= &\left\{k \in \calSin \setminus \calSin_0 : x_k \in \Rgamma \setminus \Ralpha, \mathrm{\texttt{outCG=small\_residual}}, \|s_k\| > \smallstep \right\}, \\
		\calSin_3^{s} &= &\left\{k \in \calSin \setminus \calSin_0 : x_k \in \Rgamma \setminus \Ralpha, \mathrm{\texttt{outCG=small\_residual}}, \|s_k\| \le \smallstep \right\}.
	\end{eqnarray*}	
	First consider $k \in \calSin_0$. We obtain through Lemma~\ref{le:tcg_boundary} that
	\begin{equation}
	\label{eq:inexdecS0}
		f(x_k)-f(x_{k+1}) 
		\ge \dfrac{\eta_1}{4} \gamma\Delta_k^2 \ge \dfrac{\eta_1}{4}\gamma\Deltaminin^2.
	\end{equation}	
For $k\in \calSin_1$, we have $\norm{g_k}\geq \alpha$ and Lemma~\ref{le:tcggradalpha} gives
		\begin{equation}
	\label{eq:inexdecS1}
		f(x_k)-f(x_{k+1}) 
		\ge \dfrac{\eta_1}{2}\min\left(\Delta_k, \dfrac{\alpha}{\cH}\right)\alpha
		\ge  \dfrac{\eta_1}{2} \min\!\left( \dfrac{\alpha^2}{\cH},
		\alpha\Deltaminin  \right)\!.
	\end{equation}
	 Consider now $k \in \calSin_2$, Lemma~\ref{le:meodecrease} gives
	\begin{equation}
	\label{eq:inexdecS2}
		f(x_k)-f(x_{k+1}) 
		\ge \dfrac{\eta_1}{4} \beta \Delta_k^2 	\ge  \dfrac{\eta_1}{4} \beta \Deltaminin^2,
	\end{equation}
	with probability $1-p$.\\
		For any $k \in \calSin_3^l$, Lemma~\ref{le:tcg_small_residual} guarantees that 
	\begin{equation}
	\label{eq:inexdecS3l}
		f(x_k)-f(x_{k+1}) \ge \dfrac{\eta_1}{4}\gamma\|s_k\|^2 \ge \dfrac{\eta_1\smallstep^2}{4}\gamma.
	\end{equation}
	For any $k \in \calSin_3^s$, Lemma~\ref{le:tcg_small_residual} and~\eqref{eq:decrease_small_residual_2} 
	apply. We further partition this set into $\calSin_3^{s,s} \cup \calSin_3^{s,l}$ 
	where
	\begin{eqnarray*}
		\calSin_3^{s,l} &= &\left\{k \in \calSin_{3}^s : 
		\|g_{k+1}\| \ge \min\!\left(\cquadin\min (1,\gamma,\gamma^2,\gamma\delta),
		\gamma\Delta_k\right)\! \right\}, \\
		\calSin_3^{s,s} &= &\calSin_3^s \setminus \calSin_3^{s,l}.
	\end{eqnarray*}
	If $k \in \calSin_3^{s,l}$, Equation~\eqref{eq:decrease_small_residual_2} implies that 
	\begin{eqnarray}
	\label{eq:inexdecS3sl}
		f(x_k)-f(x_{k+1})
		&\ge 
		&\dfrac{\eta_1}{2(\cR^2 + 2\hat{L}_H\cR)} \min\left(\norm{g_{k+1}}^2  \gamma^{-1}, \gamma^3 \right) 
		\nonumber \\
		&\ge 
		&\dfrac{\eta_1}{2(\cR^2 + 2\hat{L}_H\cR)}\min\left( 
		\cquadin^2 \min(\gamma^{-1},\gamma,\gamma^3,\gamma\delta^2),\gamma\Delta_k^2,\gamma^3 \right) 
		\nonumber \\
		&\ge 
		&\dfrac{\eta_1}{2(\cR^2 + 2\hat{L}_H\cR)}\min\left( 
		\cquadin^2\min(\gamma^{-1},\gamma,\gamma^3,\gamma\delta^2),\gamma\Deltaminin^2,\gamma^3 \right).
	\end{eqnarray}
	Finally, if $k \in \calSin_3^{s,s} $, we discuss based on the region of $x_{k+1}$. If $x_{k+1} \in \Rgamma$, the 
	local phase begins at $x_{k+1}$ per Theorem~\ref{theo:local_conv_inexact}. If $x_{k+1} \in \Ralpha\setminus\Rgamma$, we have $\norm{g_{k+1}}\geq \alpha$ and $k+1\in \calSin_0 \cup \calSin_1$. Finally, if $x\kplus \in \Rbeta$, it must be that $k+1 \in \calSin_0 \cup \calSin_2$. Thus,
	\begin{equation}
	\label{eq:inexbndS3ss}
		\left| \calSin_3^{s,s} \right| 
		\le \left| \calSin_0\right| + 
		\left| \calSin_1\right| + \left|\calSin_2\right| + 1 
		+ \log_2 \log_2 \left( \dfrac{2\ugamma^2}{\cR (\lHhat+2)\varepsilon_g}\right),
	\end{equation}
	and it suffices to bound $|\calSin_0|$,  $|\calSin_1|$, $|\calSin_2|$, $|\calSin_3^{l}|$, and 
	$|\calSin_3^{s,l}|$ to bound $|\calSin|$. The rest of the proof is similar to the proof of Theorem~\ref{thm:successful_complexity_exact} for the exact case. 
	The following inequality holds by~\aref{assu:lower_bound}:
	\begin{eqnarray*}
		f(x_0)-f^* &\ge &\sum_{k \in \calSin_0} f(x_k)-f(x_{k+1}) 
		+ \sum_{k \in \calSin_1} f(x_k)-f(x_{k+1}) 
		+\sum_{k \in \calSin_2} f(x_k)-f(x_{k+1}) \\
		& &+ \sum_{k \in \calSin_3^l} f(x_k)-f(x_{k+1}) 
		+\sum_{k \in \calSin_3^{s,l}} f(x_k)-f(x_{k+1}) .
	\end{eqnarray*}
	Combining this inequality with~\eqref{eq:inexdecS0},~\eqref{eq:inexdecS1}, \eqref{eq:inexdecS2}, 
	\eqref{eq:inexdecS3l} and~\eqref{eq:inexdecS3sl}, and considering each sum individually, 
	we obtain 
	\begin{eqnarray*}
		\left|\calSin_0\right| 
		&\le 
		&\frac{4(f(x_0)-f^*)}{\eta_1}\gamma^{-1}\Deltaminin^{-2}, \\
		\left|\calSin_1\right|
		&\le 
		&\frac{2(f(x_0)-f^*)}{\eta_1}\max\!\left(
		\cH\alpha^{-2},\alpha^{-1}\Deltaminin^{-1}\right), \\
		\left|\calSin_2\right|
		&\le 
		&\frac{4(f(x_0)-f^*)}{\eta_1}\beta^{-1}\Deltaminin^{-2}, \\
		\left|\calSin_3^l\right| 
		&\le 
		&\frac{4(f(x_0)-f^*)}{\eta_1} \smallstep^{-2} \gamma^{-1}, \\
		\left|\calSin_3^{s,l}\right| 
		&\le 
		&\frac{2(\cR^2 + 2\hat{L}_H\cR)(f(x_0)-f^*)}{\eta_1} 
		\max\!\left((\cquadin)^{-2}\min\!\left(\gamma^{-1},\gamma,\gamma^3,\gamma\delta^2\right)^{-1},
		\gamma^{-1}\Deltaminin^{-2},\gamma^{-3}\right)\!. 		
	%	
	%	\max\left\{(\cquadin)^{-1}\max\{\gamma,1,\gamma^{-1},\delta^{-1}\},\gamma^{-3},
	%	\Deltaminin^{-1}\right\}.
	\end{eqnarray*}
	Using the definition of $\Deltaminin$~\eqref{eq:tr_radius_inexact} as well as $\ualpha$, 
	$\ubeta$,$\ugamma$,$\udelta$, we obtain
	\begin{eqnarray*}
		\left|\calSin_0\right|
		&\le 
		&\frac{4(f(x_0)-f^*)}{\eta_1}\max\!\left(\Delta_0^{-2},1\right)\!\min \!\left(\ualpha^{4/3}\ugamma,\ubeta^{2}\ugamma,\ugamma^{3}\right)^{-1}, 
		\\
		\left|\calSin_1\right|
		&\le 
		&\frac{2(f(x_0)-f^*)}{\eta_1}\max\!\left(\cH,(\cDeltain\Delta_0)^{-1},(\cDeltain)^{-1}\right) \\
		&
		&\times
		\min\!\left(\ualpha^{2},\ualpha\ubeta ,\ualpha \ugamma\right)^{-1}, \\
		\left|\calSin_2\right|
		&\le 
		&\frac{4(f(x_0)-f^*)}{\eta_1}\max\!\left(\Delta_0^{-2},1\right)\min\!\left(\ualpha^{4/3}\ubeta,\ubeta^{3},\ugamma^{2}\ubeta\right)^{-1}, 
		\\
		\left|\calSin_3^l\right| 
		&\le 
		&\frac{4(f(x_0)-f^*)}{\eta_1} \smallstep^{-2} \ugamma^{-1}, \\
		\left|\calSin_3^{s,l}\right| 
		&\le 
		&\frac{2(\cR^2 + 2\hat{L}_H\cR)(f(x_0)-f^*)}{\eta_1} 
		\max\!\left((\cDeltain\Delta_0)^{-2},(\cquadin)^{-2}\right) \\
		&
		&\times 
		\min\!\left(\ualpha^{4/3}\ugamma,\ubeta^2 \ugamma,\ugamma^{3},\ugamma\udelta^2\right)^{-1}.
	\end{eqnarray*}	
	Combining these bounds with~\eqref{eq:inexbndS3ss} gives
	\begin{eqnarray*}
		\left| \calSin \right| 
		&=
		& \left| \calSin_0 \right| + \left| \calSin_1 \right| +\left| \calSin_2 \right| +\left| \calSin_3^l \right| 
		+\left| \calSin_3^{s,l} \right| +\left| \calSin_3^{s,s} \right| \\
		&\le
		&2\left| \calSin_0 \right| + 2\left| \calSin_1 \right| +2 \left| \calSin_2 \right|+\left| \calSin_3^l \right| +
		\left| \calSin_3^{s,l} \right| 
		+1+\log_2 \log_2 \left( \dfrac{2\ugamma^2}{\cR (\lHhat+2)\varepsilon_g}\right) \\
		&\le 
		&\Cinex\min\!\left(\ualpha^2,\ualpha^{4/3}\ubeta,\ualpha^{4/3}\ugamma,\ubeta^3,
		\ubeta\,\ugamma^2,\ubeta^2\ugamma,\ugamma^3,
		\ugamma\udelta^2 \right)^{-1} \\
		&
		&+1+\log_2 \log_2 \left( \dfrac{2\ugamma^2}{\cR (\lHhat+2)\varepsilon_g}\right),
	\end{eqnarray*}
	with
	\begin{eqnarray*}
		\Cinex
		&:=
		&\tfrac{4(f(x_0)-f^*)}{\eta_1}\left[ 
		\max\!\left(\Delta_0^{-2},1\right)\!+
		\max\!\left(\cH,(\cDeltain\Delta_0)^{-1},(\cDeltain)^{-1}\right)
		+\smallstep^{-2} \right.\\
		&
		&\left. 
		+\tfrac{(\cR^2 + 2L\cR)}{2}\max\!\left((\cDeltain\Delta_0)^{-2},(\cquadin)^{-2}\right)
		\right].
	\end{eqnarray*}
	Thus we have $K \le K_{\calSin}$ if none of the calls to the MEO 
	with $x_k \in \Rbeta$ fail to find sufficient negative curvature. The probability of all 
	calls succeeding is bounded below by 
	$(1-p)^{|\calSin_2|} \ge (1-p)^K \ge (1-p)^{K_{\calSin}}$, hence the conclusion.
\end{proof}

As in the exact setting, we can express the number of total iterations as a function of 
the number of successful iterations. This leads to the following result.

\begin{theorem}\label{thm:complexity_inexact}
	Under the assumptions of Theorem~\ref{thm:successful_complexity_inexact},
	Algorithm~\ref{algo:inexact_RTR} produces a point that satisfies~\eqref{eq:target} in at most
	\begin{small}
	\begin{eqnarray*}
		\tilde{K} 
		= \frac{{1+\log_{\tau_2}(1/\tau_1)}}{\log_{\tau_2}(1/\tau_1)}K_{\calSin}
		+
		\frac{1}{\log_{\tau_2}(1/\tau_1)}
		\max\left(0,
		\log_{\tau_2}\left(\dfrac{1}{\cDeltain}\right),
		\log_{\tau_2}\left(\dfrac{\Delta_0}{\cDeltain\ualpha^{\tfrac{2}{3}}}\right),
		\log_{\tau_2}\left(\dfrac{\Delta_0}{\cDeltain \ubeta}\right),
		\log_{\tau_2}\left(\dfrac{\Delta_0}{\cDeltain \ugamma}\right)
		\right)
	\end{eqnarray*}
	\end{small}
	iterations with probability at least $(1-p)^{\tilde K}$.
%	, where $K$ is defined in 
%	Theorem~\ref{thm:successful_complexity_inexact}. 
\end{theorem}
\begin{proof}
	Let $K \in \N$ be an iteration index satisfying the same assumptions than in the proof of 
	Theorem~\ref{thm:successful_complexity_inexact}, and let $\calSin$ be defined according to 
	$K$ as in that proof.
	By the same argument as in Lemma~\ref{le:unsuccessful_complexity_exact}, 
	we have that
	\begin{eqnarray*}
	\label{eq:succtotalinex}
		|\calSin|
		&\ge
		&\frac{\log_{\tau_2}(1/\tau_1)}{1+\log_{\tau_2}(1/\tau_1)}(K+1) \nonumber \\
		&
		&- \frac{1}{1+\log_{\tau_2}(1/\tau_1)}
		\max\left[0,
		\log_{\tau_2}\left(\dfrac{1}{\cDeltain}\right),
		\log_{\tau_2}\left(\dfrac{\Delta_0}{\cDeltain\alpha^{\tfrac{1}{2}}}\right),
		\log_{\tau_2}\left(\dfrac{\Delta_0}{\cDeltain\alpha^{\tfrac{2}{3}}}\right),
		\log_{\tau_2}\left(\dfrac{\Delta_0}{\cDeltain \beta}\right), 
		\right. \\
		&
		&\left.
		\log_{\tau_2}\left(\dfrac{\Delta_0}{\cDeltain \gamma}\right)
		\right].
	\end{eqnarray*}
	Combining this result with Theorem~\ref{thm:complexity_inexact} gives the desired 
	bound, that holds with probability at least $(1-p)^{\tilde{K}}$ if Algorithm~\ref{algo:MEO} 
	can be called at every iteration, and $(1-p)^{K_{\calSin}}$ if it is called only 
	on successful iterations.
\end{proof}

\begin{remark}
\label{rem:probainex}
	While implementing Algorithm~\ref{algo:inexact_RTR}, one may avoid repeated calls to 
	the MEO if a negative curvature direction was found at iteration $k$ and the 
	iteration is unsuccessful~\cite[Implementation strategy 4.1]{curtis2021trust}. Such an approach 
	can improve the probability bounds in Theorems~\ref{thm:successful_complexity_inexact} 
	and~\ref{thm:complexity_inexact} to $(1-p)^{K_{\calSin_2}}$, where $K_{\calSin_2}$ is a bound on 
	the number of iterations in $\calSin_2$ as defined in the proof of the theorem. In this paper, 
	we focus on the improvement brought by the strict saddle properties of $f$, and therefore use a 
	simpler, albeit suboptimal bound on the probability.
\end{remark}

Since Algorithm~\ref{algo:inexact_RTR} relies on an iterative procedure to solve the subproblems, we also provide 
a bound on the number of Hessian-vector products required to reach an approximate stationary point. 
The bound is obtained by accounting for both the inner iterations of tCG 
and the cost of the MEO. The former is bounded by $\kmax$ in the 
algorithm, while the latter follows directly from Assumption~\ref{assu:iter_meo}. Combining these 
properties with the result of Theorem~\ref{thm:complexity_inexact}, we obtain the following result.

\begin{corollary}%[Complexity of Algorithm~\ref{algo:inexact_RTR}]
\label{coro:Hvinex}
	Under the assumptions of Theorem~\ref{thm:complexity_inexact}, the total number 
	of Hessian-vector products performed by Algorithm~\ref{algo:inexact_RTR} and its 
	subroutines before producing an iterate satisfying~\eqref{eq:target} is at most
	\begin{small}
	\begin{eqnarray}
	\label{eq:Hvinex}
		\max\!\left(\kmax,\Nmeo\right)\!\tilde{K}  = \mathcal{O}\left(
		 \max\!\left(n,\ln\left(\max\!\left(\varepsilong^{-2}\ugamma^{-1/2},
		 \varepsilong^{-1}\ugamma^{-1/2},\ugamma^{-3/2}\right)\!\right)\ugamma^{-1/2},
		 \ln(n)\ubeta^{-1/2}\right)\!
		 \right. \nonumber \\
		 \left.
		 \times   
		 \left[
		 \min\!\left(\ualpha^2,\ualpha^{4/3}\ubeta,\ualpha^{4/3}\ugamma,\ubeta^3,
		\ubeta\,\ugamma^2,\ubeta^2\ugamma,\ugamma^3,
		\ugamma\udelta^2 \right)^{-1}
		 +
		 \log_2\log_2\left(\frac{\ugamma^2}{\varepsilon_g}\right)
		 \right]
		\right)
	\end{eqnarray}
	\end{small}
	with probability $(1-p)^{\tilde{K}}$, where $\tilde K$ is defined in Theorem~\ref{thm:complexity_inexact}.
\end{corollary}

As a final comment, we note that the operation complexity bound~\eqref{eq:Hvinex} 
exhibits an additional logarithmic factor in $\varepsilon_g$ compared to the iteration 
complexity bound of Theorem~\ref{thm:complexity_inexact}, but no additional dependency 
on $\varepsilon_H$, unlike in the case of general nonconvex functions. Overall, the 
complexity guarantees of inexact variants also improve thanks to the strict saddle 
property.

%%%%%%%%%%%%%%%%%%%%%%%%%%%%%%%%%%%%%%%%%%%%%%%%%%%%%%%%%%%%%%%%%%%%%%%%%%%%%%%%%%%%%%%%%
\section{Conclusions}
\label{sec:conc}
%%%%%%%%%%%%%%%%%%%%%%%%%%%%%%%%%%%%%%%%%%%%%%%%%%%%%%%%%%%%%%%%%%%%%%%%%%%%%%%%%%%%%%%%%

We show that worst-case complexity guarantees of Riemannian trust-region
algorithms on nonconvex functions improve significantly when the function satisfies a
strict saddle property. The guarantees we provide only depend logarithmically on the 
prescribed optimality tolerances, and, as such, are a better reflection of how 
problem-dependent quantities affect the performance. In particular, an algorithm with 
exact subproblem minimization does not require any modification from its standard 
version in order to benefit from these improved guarantees. Our analysis relies on the local quadratic convergence of Newton's method, and can be adapted to 
inexact subproblem minimizations by incorporating knowledge of the strict saddle constants in 
the problem. Although those parameters are known for a variety of problems, adaptive 
schemes have been proposed to estimate them as the algorithm 
unfolds~\citep{oneill2023linesearch}. Investigating the numerical performance of these 
algorithms, along with their multiple possibilities for implementation, will be the 
subject of future work. Extending our results to a broader class of strict saddle functions which includes non-isolated minimizers would also be valuable.

\paragraph{Acknowledgments} The authors are grateful to Coralia Cartis for useful discussions
regarding local convergence of trust-region methods and for sharing reference~\citep{shek2015master}.
This research was partially funded by the Agence Nationale de
la Recherche through program ANR-19-P3IA-0001 (PRAIRIE 3IA Institute).

\bibliographystyle{apalike}
 \bibliography{saddle}
\end{document}